\documentclass[a4paper, 11pt]{amsart}
\usepackage[latin1]{inputenc}
\usepackage[T1]{fontenc}
\usepackage[english]{babel}
\usepackage{amssymb}
\usepackage{amsmath}
\usepackage{amsthm}
\usepackage{amscd}
\usepackage{amsfonts}
\usepackage{stmaryrd}
\usepackage{pb-diagram}
\usepackage{epic,eepic,epsfig}
\usepackage{a4wide}
\usepackage{nextpage}
\usepackage{fancyhdr}

\newtheorem{prop}{Proposition}[section]
\newtheorem{cor}[prop]{Corollary}

\newtheorem{lem}[prop]{Lemma}
\newtheorem{thm}[prop]{Theorem}
\newtheorem{conj}[prop]{Conjecture}

\newtheorem{theorem}[prop]{Theorem}
\newtheorem{corol}[prop]{Corollary}

\theoremstyle{definition}
\newtheorem{rem}[prop]{Remark}
\newtheorem{defin}[prop]{Definition}

\newcommand{\Jac} {\mathop{\mathrm{Jac}}}
\newcommand{\Ima} {\mathop{\mathrm{Im}}}

\newcommand{\Ree} {\mathop{\mathrm{Re}}}
\newcommand{\Card} {\mathop{\mathrm{Card}}}

\newcommand{\degr} {\mathop{\mathrm{deg}}}

\newcommand{\Spec} {\mathop{\mathrm{Spec}}}

\newcommand{\Lie} {\mathop{\mathrm{Lie}}}

\begin{document}

\title{Theta height and Faltings height}
\author{Fabien Pazuki}
\maketitle
\vspace{-0.9cm}
\begin{center}
({\small after {\sc J.--B. Bost} and {\sc S. David}})\\
\vspace{0.2cm}
\today
\end{center}

\begin{abstract}
Using original ideas from {\sc J.--B. Bost} and {\sc S. David}, we provide an explicit comparison between the Theta height and the stable {\sc Faltings} height of a principally polarized abelian variety. We also give as an application an explicit upper bound on the number of $K$--rational points of a curve of genus $g\geq 2$ under a conjecture of {\sc S. Lang} and {\sc J. Silverman}. We complete the study with a comparison between differential lattice structures.
\end{abstract}

{\flushleft
\textbf{Keywords~:} Heights, Abelian varieties, Rational points.\\
\textbf{Mathematics Subject Classification~:} 11G50, 14G40, 14G05. }

\begin{center}
---------
\end{center}

\begin{center}
\textbf{Hauteur Th\^eta et hauteur de {\sc Faltings}.}
\end{center}

\begin{abstract}
On propose dans cet article les d\'etails d'une preuve de comparaison explicite entre la hauteur Th\^eta et la hauteur de {\sc Faltings} stable d'une vari\'et\'e ab\'elienne principallement polaris\'ee et d\'efinie sur un corps de nombres $K$. Cette preuve est bas\'ee sur les id\'ees de {\sc J.-B. Bost} et {\sc S. David}. On trouvera de plus le calcul d'une borne explicite sur le nombre de points $K$-rationnels d'une courbe de genre $g\geq 2$ en supposant une conjecture de {\sc S. Lang} et {\sc J. Silverman}. Ce travail est compl\'et\'e par une comparaison entre plusieurs structures de r\'eseaux sur l'espace tangent en $0$.
\end{abstract}

{\flushleft
\textbf{Mots-Clefs~:} Hauteurs, Vari\'et\'es ab\'eliennes, Points rationnels.\\
}

\begin{center}
---------
\end{center}

\thispagestyle{empty}

\section{Introduction}
Let $(A,L)$ be a principally polarized abelian variety defined over a number field $K$. The aim of the article is to compare the Theta height $h_{\Theta}(A,L)$ of definition \ref{ThetaHeight}, and the (stable) {\sc Faltings} height $h_{F}(A)$ of definition \ref{FaltHeight}. These two ways of defining the height of an abelian variety are both of interest, and the fact that they can be precisely compared can be very helpful. For instance, several conjectures are formulated with the {\sc Faltings} height because it does not depend on the projective embedding of $A$ that you may choose, but one may fix an ample and symmetric line bundle on $A$ and study the Theta height associated when one seeks more effectivity (see for example \cite{DavPhi} or \cite{Rem2}, and also \cite{Rem1}); let us stress that these ways of defining the height of an abelian variety are very natural: the Theta height is a height on the moduli space of principally polarized abelian variety and the {\sc Faltings} height is a height on the moduli space (stack) of abelian varieties (without polarization), but with a metric with logarithmic singularities (see the definitions below and refer to \cite{Mum1} for the Theta height, \cite{Moret3} and \cite{FalCha} for the {\sc Faltings} height). 

The ideas needed to explicitly compute the constants of comparison between these heights were given by {\sc Bost} and {\sc David} in a letter to \textsc{Masser} and \textsc{Wüstholz} \cite{BoDa}. Here is the strategy: using the theory of \textsc{Moret--Bailly}--models we express the \textsc{Néron--Tate} height of a point $P\in{A(K)}$ in terms of the Theta height of $P$, the \textsc{Faltings} height of $A$ and some base point contributions (see lemma \ref{hnt}). Then we take $P=O$ and we estimate the base point contributions via vector bundles inclusions and theta functions analysis.  We give here the arguments, the constants and several complements, concerning the \textsc{Lang--Silverman} conjecture for instance. We also complete this work by giving in section~\ref{diff lattice} an explicit comparison between several differential lattice structures associated to $A$, see the end of this introduction.

One should underline that this explicit comparison gives also a direct proof of the fact that the {\sc Faltings} height is actually a height (\textit{i.e.} verifies the Northcott property), see the remark \ref{minoration hauteur de Faltings} below for a lower bound. Arguments for proving that $h_{F}$ is a height can be found in the original article \cite{Falt} and in \cite{FalCha}. See also the Theorem 1.1 page 115 of \cite{Moret3} (seminar \cite{Szp}); the idea is to compactify some moduli schemes and to compare the stable {\sc Faltings} height of an abelian variety to the projective height (with logarithmic singularities) of the corresponding point in the moduli space. There is another proof given by {\sc Moret--Bailly} in Theorem 3.2 page 233 of \cite{Moret} using the ``formule clef'' 1.2 page 190. See also the Theorem 2.1 given in \cite{Bost1} page 795--04, where the proof relies on some estimates of the ``rayon d'injectivité''.

The author thanks {\sc J.--B. Bost} and {\sc S. David} for sharing their ideas, for their support and helpful comments, {\sc A. Chambert-Loir} and {\sc G. Rémond} for their interest.

\vspace{0.2cm}

We use the notations ${\mathfrak S}_{g}$ for the {\sc Siegel} space and ${\mathfrak F}_{g}$ for the fondamental domain, both defined in \S \ref{FundamentDomain}. We add a Theta structure of level $r$ (see \S \ref{thetastruc}), where $r>0$ is an even integer. With these notations, we get the following theorem.
\begin{theorem}
\label{hautfalhautthet}Let $A$ be an abelian variety  of dimension $g$, defined over
$\overline{\mathbb{Q}}$, equipped with a principal polarization defined by a symmetric ample line bundle $L$ on
$A$. Let $K$ be a number field such that $A$ and $L$ may be defined over $K$.
For any
embedding $\sigma \,:\, K\hookrightarrow \mathbb{C}$, let $\tau_{\sigma}\in {\mathfrak F}_{g}$ such that there
exists an isomorphism between principally polarized complex abelian varieties $A_{\sigma}(\mathbb{C})\simeq\mathbb{C}^g/(\mathbb{Z}^g+\tau_{\sigma}\mathbb{Z}^g).$
Then, the following inequalities hold~:
$$m(r,g)\leq h_{\Theta}(A,L)-\frac{1}{2}
h_F(A)-\frac{1}{4[K:\mathbb{Q}]}\sum_{\sigma:K\hookrightarrow\mathbb{C}}\log(\det(\Ima\tau_{\sigma}))\leq
M(r,g)\;.$$
Above, $m(r,g)$ and $M(r,g)$ denote constants depending only on the level $r$ and the dimension $g$.
More precisely, if we take~:
\[
m(r,g)=g\left[\frac{1}{4}\log(4\pi)-\frac{1}{2} r^{2g}\log(r)\right],\;\; M(r,g)=\frac{g}{4}\log(4\pi)+g\log(r)+\frac{g}{2}\log\left(2+
\frac{2}{3^{\frac{1}{4}}}2^{\frac{g^3}{4}}\right)
\]
the result holds.
\end{theorem}

\begin{rem}\label{matrix lemma} According to the so called Matrix Lemma of {\sc Masser} (see \cite{Masser} page 115 or \cite{MaWu} page 436) there exists a constant $C(g)$ such
that under the hypothesis of the above theorem~:
$$\frac{1}{[K:\mathbb{Q}]}\sum_{\sigma: K\hookrightarrow \mathbb{C}}\Big|\log\Big(\det(\Ima\tau_{\sigma})\Big)\Big|\leq
C(g)\log\Big(\max\{h_{\Theta}(A,L),1\}+2\Big)\;.$$
Using the article \cite{DavPhi} page 697 and a few calculations it is possible to prove such a bound with the explicit constant $C(g)=\frac{8g}{\pi}(1+2g^2\log(4g))$. See also \cite{Graf} lemma 2.12 page 99 for a similar statement involving the {\sc Faltings} height.
\end{rem}

Thus, we shall establish in \S~\ref{ProofMainThm} the following versions of {\sc Faltings'}
estimate (see \cite{Falt})~:
\begin{corol}
\label{varcomphaut} For every integer $g\geq 1$ and even integer $r\geq 2$, there exists effectively computable constants $C_{1}(g,r)$, $C_{2}(g,r)$, $C_{3}(g,r)$ depending only on $g$ and $r$ such that the following holds.
Let $A$ be an abelian variety of dimension $g$ defined over $\overline{\mathbb{Q}}$, equipped with a
principal polarisation defined by some symmetric ample line bundle $L$ on $A$. Let $h_{\Theta}=\max\{h_{\Theta}(A,L),1\}$ and $h_{F}=\max\{h_{F}(A),1\}$. Then, one has~:
\begin{enumerate}
\item $\displaystyle{\Big\vert h_{\Theta}(A,L)-\frac{1}{2} h_{F}(A)\Big\vert\leq
C_{1}(g,r)\log\Big(h_{\Theta}+2\Big),}$
\vspace{0.2cm}
\item $\displaystyle{\Big\vert h_{\Theta}-\frac{1}{2} h_{F}\Big\vert\leq
C_{2}(g,r)\log\Big(\min\Big\{h_{\Theta},h_{F}\Big\}+2\Big),}$
\vspace{0.2cm}
\item $\displaystyle{\Big\vert h_{\Theta}(A,L)-\frac{1}{2} h_{F}'(A)\Big\vert\leq C_{3}(g,r),}$
\end{enumerate}
where $h_{F}'(A)$ is a modified {\sc Faltings} height of $A$, defined in \ref{modifiedFaltHeight}. More precisely, the above relations hold with~:
\[
C_{1}(g,r)=C_{3}(g,r)=6r^{2g}\log(r^{2g}) \;\,\textrm{and}\;\, C_{2}(g,r)=1000r^{2g}(\log(r^{2g}))^5.
\]

\end{corol}

\begin{rem}\label{minoration hauteur de Faltings}
For an abelian variety $A$ of dimension $g$ and level structure $r$, the inequality of Theorem \ref{hautfalhautthet} and the remark \ref{matrix lemma} give after a short calculation~:
\[
h_{F}(A)\geq -C(g)\log C(g)\,-\,M(r,g)\;,
\]
where $M(r,g)=\frac{g}{4}\log(4\pi)+g\log(r)+\frac{g}{2}\log\Big(2+
\frac{2}{3^{\frac{1}{4}}}2^{\frac{g^3}{4}}\Big)$ and $C(g)=\frac{8g}{\pi}\left(1+2g^2\log(4g)\right)$. One could expect a better constant, see {\sc Bost} in \cite{Bost3} page 6 who gives: $h_{F}(A)\geq -g\log(2\pi)/2$.
\end{rem}

\begin{rem}
The inequalities $(1)$ and $(3)$ both hold if one replaces $h_{\Theta}(A,L)$, $h_{F}(A)$ and $h_{F}'(A)$ respectively by $h_{\Theta}=\max\{h_{\Theta}(A,L),1\}$, $h_{F}=\max\{h_{F}(A),1\}$ and $h_{F}'=\max\{h_{F}'(A),1\}$ in the left hand sides.
\end{rem}

\begin{rem}
One can notice that the bounds are sharper for small $r$, so in practice one will often take $r=2$ or $r=4$.
\end{rem}

We now give the example of a difficult conjecture by {\sc Lang} and {\sc Silverman} stated with the {\sc Faltings} height. It was originally a question by {\sc Lang} concerning elliptic curves, and was generalised by {\sc Silverman} afterwards. As a matter of fact, if we combine the inequality of this conjecture with the work of {\sc David} and {\sc Philippon} \cite{DavPhi} and the work of {\sc Rémond} \cite{Rem2}, we get a new explicit bound on the number of rational points on curves of genus $g\geq 2$, provided that we can explicitely compare the {\sc Faltings} height that appears in the conjecture and the Theta height that appears in the calculations of \cite{DavPhi} and \cite{Rem2}. To be concise, one can say that an explicit {\sc Lang--Silverman} inequality would give an explicit upper bound on the number of rational points on a curve of genus $g\geq 2$ independant of the height of the jacobian of the curve (but still depending on the Mordell--Weil rank of the jacobian).

First recall the original conjecture of {\sc Silverman} (\cite{Sil3} page 396)~:

\begin{conj}({\sc Lang--Silverman} version 1)\label{LangSilV1}
Let $g\geq 1$ be an integer. For any number field $K$, there exists a positive constante $c(K,g)$ such that for any abelian variety $A/K$ of dimension $g$, for any ample and symmetric line bundle $L$ on $A$ and for any point $P\in{A(K)}$ such that $\mathbb{Z}\!\cdot\! P$ is Zariski--dense, one has~:
\[
\widehat{h}_{A,L}(P) \geq c(K,g)\, \max\Big\{h_{F}(A/K),\,1\Big\}\;,
\]
where $\widehat{h}_{A,L}(.)$ is the {\sc N\'eron--Tate} height associated to the line bundle $L$ and $h_{F}(A/K)$ is the (relative) {\sc Faltings} height of the abelian variety $A/K$.
\end{conj}

One could read \cite{Paz2} for further remarks.
Let us give a slightly different version of this conjecture. The definition of the modified {\sc Faltings} height is given in \ref{modifiedFaltHeight}~:

\begin{conj}({\sc Lang--Silverman} version 2)\label{LangSilV2}
Let $g\geq 1$ be an integer. For any number field $K$ of degree $d$, there exists two positive constants $c_{1}=c_{1}(d,g)$ and $c_{2}=c_{2}(d,g)$ such that for any abelian variety $A/K$ of dimension $g$ and any ample symmetric line bundle $L$ on $A$, for any point $P\in{A(K)}$, one has~:
\begin{itemize}
\item[$\bullet$] either there exists a sub-Abelian variety $B\subset A$, $B\neq A$, of degree $deg(B)\leq c_{2}$ and such that the point $P$ is of order bounded by $c_{2}$ modulo $B$,
\item[$\bullet$] or one has $\mathbb{Z}\!\cdot\! P$ is Zariski--dense and~:
\[
\widehat{h}_{A,L}(P) \geq c_{1}\, \max\Big\{h_{F}'(A),\,1\Big\}\;,
\]
where $\widehat{h}_{A,L}(.)$ is the {\sc N\'eron--Tate} height associated to the line bundle $L$ and $h_{F}'(A)$ is the (stable) {\it modified} {\sc Faltings} height of the abelian variety $A$.
\end{itemize}
\end{conj}

This second version of the conjecture is suggested by different results found in \cite{Dav2} and \cite{Paz}. Note that one could also state it with a relative modified {\sc Faltings} height (that would be a stronger statement). Now this second version and the point $(3)$ in Corollary~\ref{varcomphaut} give, if we use them in the work of {\sc David--Philippon} \cite{DavPhi} and {\sc Rémond} \cite{Rem2} (see {\it infra} \S \ref{ProofMainThm} for some details)~:

\begin{prop}\label{ExplicitBound}
Assume conjecture \ref{LangSilV2}. Then for any number field $K$, for any curve $C/K$ of genus $g\geq 2$ with jacobian $J=\Jac(C)$, one can explicitely bound the number of $K$--rational points on $C$ in the way~:
\[
\Card(C(k))\leq \Big(\breve{c}(d,g)\Big)^{1+\mathrm{rk}(J/K)}\;,
\]
where one can take $\breve{c}(d,g)=\max\Big\{2c_{2}\,,\; 1+(12^4+g)^{2^{12}}4^{2g+3}g\Big(g^4+2^{2g+2}g+\frac{1}{c_{1}}\Big)\Big\}$, with $c_{1}$ and $c_{2}$ given in conjecture \ref{LangSilV2}.
\end{prop}
One can also read \cite{DeDiego} for another way of deriving this type of bounds.
Finally, this work also includes in section \ref{diff lattice} a comparison between different lattice structures on the tangent space at $0$ of an abelian variety. Let $A/K$ be an abelian variety, $L$ an ample symmetric line bundle associated to a principal polarisation. Let $r$ be an even positive integer. By enlarging $K$, one can assume that the $r^2$--torsion points are all rational over $K$. We let $\pi \colon \mathcal{A}\to S=\Spec(\mathcal{O}_{K})$ be a semi--stable model of $A$, and $\varepsilon$ its neutral section. We define the \textsc{Néron} lattice by $\mathcal{N}=\varepsilon^{*}\Omega_{\mathcal{A}/S}^1$. The big \textsc{Shimura} lattice is defined as follows: let $\theta \in{\Gamma(A,L)\backslash \{0\}}$ and $\Gamma$, $\varphi_{x}$, \textit{etc.} be as in paragraph \ref{thetastruc}. Let $\theta_{x}=\varphi_{x}(0)$. The family $(\theta_{x})_{x\in{\Gamma}}$ is a base over $K$ of $\Gamma(A,L^{\otimes r^2})$. Then the big \textsc{Shimura} lattice is:
\[
\mathcal{S}h=\sum_{\substack{(x,x')\in{\Gamma^2}\\ \theta_{x}(0)\neq 0}}\mathcal{O}_{K} d\Big(\frac{\theta_{x'}}{\theta_{x}}\Big)(0)\;.
\]

The comparison between these two structures is of interest in transcendence theory, see for example \cite{Shim}, \cite{Dav1} page 134 or  \cite{MaWu} from page 120.
We use $\delta(.,.)$ for the distance on lattices defined in \ref{distance}. Then we find in section \ref{diff lattice}, among other results, the following theorem:

\begin{thm}
Let $g\geq 1$ and $r>0$ an even integer. There exists a constant $c(g,r)>0$ such that for any triple $(A,L,r)$ with $A$ of dimension $g$, for any associated MB\footnote{See definition \ref{bon}.} number field $K$, one has:
\[
\delta(\mathcal{N},\mathcal{S}h)\leq \; \Big(1+2c(g,r)\Big)\, \min\{h_{\Theta},h_{F}\}\;,
\]
and one can take $c(g,r)=4+8C_{2}+g\log(\pi^{-g}g!e^{\pi r^2}g^4)+4r^{2g}$, where $C_{2}$ is given in Corollary~\ref{varcomphaut}.

\end{thm}

\section{Definitions}

\subsection{Basic Notations}
\label{notbase}
Let us first introduce the following notations. If $A$ is an abelian variety defined over
a number field $K$, or, more generally an abelian scheme over a base scheme $S$, we shall
denote for any $n$ in $\mathbb{Z}$ by $[n]$~:
$$[n]\colon A\longrightarrow A\;,$$
the group scheme morphism defined by the multiplication by $n$, and when $n>0$, by $A_n$
its kernel; for any $x\in A(K)$ (respectively $A(S)$), we shall denote by $t_x$ the
morphism of $K$--variety (respectively of scheme)~:
$$t_x \colon A\longrightarrow A\;,$$
defined by the translation by $x$.

Since we shall also make an extensive use of the classical theory of theta functions
(essentially to evaluate various analytic invariants), we also recall a few basic
definitions here involving the standard {\sc Riemann} theta function.
Let $g$ be an integer $g\geq 1$. We shall denote by ${\mathfrak S}_g$ the {\sc Siegel} upper
half space, {\it i. e.} the space of $g\times g$ symmetric matrices with entries in $\mathbb{C}$,
whose imaginary part are positive definite. Let $z=x+iy\in \mathbb{C}^g$ and $\tau=X+iY\in{\mathfrak S}_g$ (in all this paper, it will be implicitly assumed that such an expansion implies
that $x$, $y$, $X$, $Y$ all have real entries). Also, unless otherwise specified, it will
be assumed that vectors in $\mathbb{C}^g$ have column entries. The classical {\sc Riemann} theta
function is then~:
$$\theta(\tau,z)=\sum_{n\in\mathbb{Z}^g}\exp\left(i\pi^t\kern-1ptn.\tau.n+2i\pi^t\kern-1ptn.z\right).$$

For $m_1,m_2\in\mathbb{R}^g$ we shall also introduce after {\sc Riemann}, {\sc Jacobi}, {\sc Igusa}, the theta functions with characteristics defined by~:
$$
\theta_{(m_1,m_2)}(\tau,z)=\kern-1pt\sum_{n\in\mathbb{Z}^g}\kern-1pt\exp\kern-.5pt\left(\kern-.14pti\pi^t\kern-1pt(n+m_1).
\tau.(n+m_1)\kern-.4pt+\kern-.4pt2i\pi^t\kern-1pt(n+m_1).(z+m_2)\kern-.14pt\right)\kern-.14pt.$$

These functions will be equipped with the following norm (made invariant with respect to the
action of the symplectic group)~:
$$\Vert \theta\Vert(\tau,z)=\det(Y)^{\frac{1}{4}}\exp\left(-\pi^t\kern-1pty.Y^{-1}.y\right)
\left\vert\theta(\tau,z)\right\vert,$$
and~:
$$\Vert
\theta_{(m_1,m_2)}\Vert(\tau,0)=\det(Y)^{\frac{1}{4}}
\left\vert\theta_{(m_1,m_2)}(\tau,0)\right\vert.$$
The above norm can similarly be defined for any $z\in\mathbb{C}^g$, but we shall only need it for $z=0$.
It should be also noted that $\theta_{(0,0)}(\tau,z)=\theta(\tau,z)$.

\vspace{0.3cm}

Let us denote by ${\mathfrak F}_{g}$ the usual fundamental domain of ${\mathfrak S}_g$ ({\it confer}
\cite{Igu}, V.~4.). Recall that it is characterized by the following properties~:
\begin{itemize}\label{FundamentDomain}
\item[{S.~1.}] If $\tau \in {\mathfrak F}_{g}$, then  for every $\gamma\in {\rm Sp}_{2g}(\mathbb{Z})$, one has
$\det(\Ima(\gamma.\tau))\leq\det(\Ima\tau)$, where if $\gamma=\left(\begin{array}{ll}
\alpha & \beta\\\lambda &\mu\end{array}\right)$,
$\gamma.\tau=(\alpha\tau+\beta)(\lambda\tau+\mu)^{-1}$.
\item[{S. 2.}] If $\tau=(\tau_{i,j})_{1\leq i,j\leq g}\in {\mathfrak F}_{g}$, 
then~:
$$\forall (i,j)\in\left\{1,\ldots,g\right\}^2, \hspace{.5cm}
\left\vert\Ree(\tau_{i,j})\right\vert\leq \frac{1}{2}\;.$$
\item[{S. 3.}] If $\tau\in {\mathfrak F}_{g}$,
\begin{itemize}
\item[$\bullet$] $\forall k\in\{1,\ldots,g\}$  and  $\forall \xi\in\mathbb{Z}^g,\hspace{.4cm}
(\xi_k,\ldots,\xi_g)=1$, one has~ $^t\kern-1pt\xi.\Ima(\tau).\xi\geq \Ima(\tau_{k,k}).$
\item[$\bullet$] $\forall k\in\{1,\ldots,g-1\}$, one has $\Ima(\tau_{k,k+1})\geq0.$
\end{itemize}
\end{itemize}

\vspace{0.3cm}

Finally, we shall also make use of the following notations for projective spaces. Assume that $E$
is a vector bundle over some noetherian scheme $S$, we shall denote by $\mathbb{P}(E)$ the scheme
${\rm Proj}({\rm Sym}(\check{E}))$ (where\footnote{The symbol $\check{E}$ stands for the dual of $E$.} ${\rm Sym}(\check{E})=\sum_{d>0}S^{d}(\check{E})$ is
the symmetric algebra of $\check{E}$). This is nothing but
$\mathbb{P}({\mathcal E})$, where
$\mathcal E$ is the sheaf of sections of the dual bundle of $E$, in {\sc Grothendieck}'s
notations (see {\it e. g.} \cite{Hart}, page~162). The canonical quotient line bundle will be
denoted by ${\mathcal O}_E(1)$.

\subsection{The {\small{\large F}ALTINGS} height}
\label{defhautfal} 
Let $A$ be an abelian variety defined over $\overline{\mathbb{Q}}$, of dimension $g$ ($g\geq 1$), and
$K$ a number field over which $A$ is rational and semi--stable. Put $S=\Spec({\mathcal O}_K)$, where
${\mathcal O}_K$ is the ring of integers of $K$. Let $\pi\colon {\mathcal A}\longrightarrow S $
be a semi--stable model of $A$ over $S$. We shall denote by $\varepsilon$ the zero section of
$\pi$, so $\varepsilon\colon S\longrightarrow {\mathcal A}$  and by
$\omega_{{\mathcal A}/S}$ the sheaf of maximal exterior
powers of the sheaf of relative differentials~:
$$\omega_{{\mathcal A}/S}:=\varepsilon^{\star}\Omega^g_{{\mathcal
A}/S}\simeq\pi_{\star}\Omega^g_{{\mathcal A}/S}\;.$$

For any embedding $\sigma$ of $K$ in $\mathbb{C}$, the corresponding line bundle~:
$$\omega_{{\mathcal A}/S,\sigma}=\omega_{{\mathcal A}/S}\otimes_{{\mathcal O}_K,\sigma}\mathbb{C}\simeq H^0({\mathcal
A}_{\sigma}(\mathbb{C}),\Omega^g_{{\mathcal A}_\sigma}(\mathbb{C}))\;$$
can be equipped with a natural $L^2$--metric $\Vert.\Vert_{\sigma}$ defined by~:
$$\Vert\alpha\Vert_{\sigma}^2=\frac{i^{g^2}}{(2\pi)^g}\int_{{\mathcal
A}_{\sigma}(\mathbb{C})}\alpha\wedge\overline{\alpha}\;$$
(note that we follow here the normalization chosen by {\sc Deligne}, in \cite{Del} or \cite{Bost1} page 795--04).

The ${\mathcal O}_K$--module of rank one $\omega_{{\mathcal A}/S}$, together with the hermitian norms
$\Vert.\Vert_{\sigma}$ at infinity defines an hermitian line bundle 
$\overline{\omega}_{{\mathcal A}/S}$ over $S$, which has a well defined {\sc Arakelov} degree
$\widehat{\degr}(\overline{\omega}_{{\mathcal A}/S})$. Recall that for any hermitian line
bundle $\overline{\mathcal E}$ over $S$, the {\sc Arakelov} degree of $\overline{\mathcal E}$
is defined as~:
$$\widehat{\degr}(\overline{\mathcal E})=\log\Card\left({\mathcal E}/{{\mathcal
O}}_Ks\right)-\sum_{\sigma\colon K\hookrightarrow \mathbb{C}}\log\Vert
s\Vert_{\sigma}\;,$$
where $s$ is  any non zero section of $\mathcal E$ (which does not depend on the choice
of $s$ in view of the product formula). More generally, when
$\overline{\mathcal E}$ is an hermitian vector bundle over $S$, one defines its {\sc
Arakelov} degree as~:
$$\widehat{\degr}(\overline{\mathcal E})=\widehat{\degr}(\det(\overline{\mathcal E}))$$
(where the metrics on $\det({\mathcal E})$ at the archimedean places are those induced by
the metrics of $\overline{\mathcal E}$).

We now give the definition of the {\sc Faltings} height that one finds in \cite{Falt} page 354. 

\begin{defin}\label{FaltHeight}  The normalized stable {\sc Faltings} height of $A$ is defined as~:
$$h_F(A):=\frac{1}{[K:\mathbb{Q}]}\widehat{\degr}(\overline{\omega}_{{\mathcal
A}/S})\;.$$
\end{defin}

This height only depends on the $\overline{\mathbb{Q}}$--isomorphism class of $A$. It is also called the differential height in \cite{Moret}. To see that it is really a height, see for instance \cite{Falt} Satz~1, page 356 and 357.
We will also define a modified {\sc Faltings} height for polarized abelian varieties, very useful in some applications.

\begin{defin}\label{modifiedFaltHeight} Let $A$ be an abelian variety. If $A$ is principally polarized, then for every embedding $\sigma:K\hookrightarrow\mathbb{C}$, choose $\tau_{\sigma}\in{\mathfrak F}_{g}$ associated with $A_{\sigma}(\mathbb{C})$. The normalized stable modified {\sc Faltings} height of $A$
is defined as follows~:
$$h_{F}'(A):=h_{F}(A)+\frac{1}{2[K:\mathbb{Q}]}\sum_{\sigma:K\hookrightarrow\mathbb{C}}\log(\det(\Ima\tau_{\sigma}))\;.$$
If $A$ is equipped with an ample symmetric line bundle $L$, choose an isogeny of {\bf minimal} degree $\varphi : A\to A_{0}$ where $A_{0}$ is principally polarized. For any $\sigma : K\hookrightarrow \mathbb{C}$ let $\tau_{\sigma,A_{0}}\in{\mathfrak{F}_{g}}$ denote the period matrix associated with $A_{0,\sigma}(\mathbb{C})$. Then take~:
$$h_{F}'(A):=h_{F}(A)+\frac{1}{2[K:\mathbb{Q}]}\sum_{\sigma:K\hookrightarrow\mathbb{C}}\log(\det(\Ima\tau_{\sigma,A_{0}}))+\frac{1}{2}\log(\deg \varphi)\;.$$
\end{defin}

\begin{rem}
In the situation of isogeneous abelian varieties $\varphi : A\to A_{0}$, the corollary 2.1.4 of Raynaud \cite{Ray} gives $h_{F}(A_{0})\leq h_{F}(A)+\frac{1}{2}\log(\deg \varphi)$, see also \cite{Falt}, lemma 5 page 358. Hence $h_{F}'(A_{0})\leq h_{F}'(A)$.
\end{rem}

\subsection{Theta structures and Theta height}
\label{defhautthet}
\label{thetastruc}
\subsubsection{Isomorphisms of line bundles}
\label{isom}
Let us assume we are given the following data. 
\begin{itemize}
\item
 Let $K$ be any field of characteristic zero;
\item Let $A$ be 
an abelian variety of dimension $g$  defined over $K$;
\item Let $L$ be a symmetric ample line bundle on $A$, rigidified at the origin, which
defines a principal polarisation on $A$;
\item  Let $r>0$ be an even integer.
\end{itemize}

We shall furthermore assume (by enlarging the base field $K$ if needed) that all the torsion points of order $r^2$ of $A$ are rational over $K$. For any positive integer $n$, let us define $A_{n}(K)$ the set of $K$--rational torsion points of order $n$.

Recall that there is a unique isomorphism $j$ (since $L$ is symmetric, see {\it e. g.}
\cite{BiLa}, Corollary~3.~6, page~34)~:
$$j\colon [r]^{\star}L\rightarrow^{\!\!\!\!\!\!\sim} L^{\otimes r^2}\;,$$
compatible with the rigidification of $L$.
This implies that for any $x\in A_r(K)$, there is a canonical isomorphism~:
$$i_x\colon t_x^{\star}L^{\otimes r^2}\rightarrow^{\!\!\!\!\!\!\sim} L^{\otimes r^2}\;.$$
Indeed,
 we have~:
$$t_x^{\star}L^{\otimes r^2}\stackrel{j^{-1}}{\simeq} t_x^{\star}([r]^{\star}L)\simeq([r]\circ
t_x)^{\star}L\simeq [r]^{\star}L\stackrel{j}{\simeq} L^{\otimes r^2}\;.$$

These isomorphisms $(i_x)_{x\in A_r(K)}$ are compatible, namely, for any $(x,y)\in A_r(K)^2$, the
composite map~:
\begin{equation}
\label{compat}t_{x+y}^{\star} L^{\otimes r^2}\simeq t_x^{\star}\left(t_y^{\star}L^{\otimes
r^2}\right)\stackrel{ t_x^{\star}i_y}{\simeq} t_x^{\star}L^{\otimes
r^2}\stackrel{i_x}{\simeq} L^{\otimes r^2}
\end{equation}
 coincides with $i_{x+y}$.

As a matter of fact, for any $x\in A_{r^2}(K)$, there is still an isomorphism~:
$$i_x\colon t_x^{\star}L^{\otimes r^2}\rightarrow^{\!\!\!\!\!\!\sim} L^{\otimes r^2}\;,$$
see for example \cite{BiLa}, Lemma~4.~7 {\it (c)}, page~38.

Moreover, these $i_x$ with $x\in A_{r^2}(K)$ may be chosen in such a way that for any $(x,y)\in
A_{r^2}(K)^2$, the map~(\ref{compat}) coincides with $i_{x+y}$ up to a multiplication by some
$r^2$--th roots of unity. Any such choice of the $i_x$'s will be called a {\em good choice}. If
$(i'_x)_{x\in A_{r^2}(K)}$ is another good choice, there exists a system $(\lambda_x)_{x\in
A_{r^2}(K)}$ of $r^4$--th roots of unity in $K$ (the $r^4$--th roots of unity are $K$--rational by the Weil--pairing properties, see for example Corollary~8.1.1 page 98 of \cite{Sil1}) such that for any $x\in A_{r^2}(K)$, we have
$i'_x=\lambda_x i_x$.

All these remarks easily follow from the theorem of the cube and are immediate consequences of
{\sc Mumford}'s theory of theta structures (note that any theta structure on $L^{\otimes r^2}$
induces a good choice of the $i_x$'s). See for instance \cite{Mum1} and \cite{Mum2}.

\subsubsection{Bases for $\Gamma(A,L^{\otimes r^2})$}
\label{basesection}

Let us suppose from now on that a good choice for the $(i_{x})$ has been made. For any $x\in{A_{r^2}(K)}$ let~:
$$\begin{array}{llcl}\displaystyle
\varphi_x\colon &\displaystyle \Gamma(A,L) & \displaystyle\longrightarrow 
&\displaystyle
\Gamma(A,L^{\otimes r^2})\\ & \displaystyle s & \displaystyle\longmapsto &\displaystyle
(i_x\circ t_x^{\star}\circ j\circ [r]^{\star})s\;.
\end{array}$$

It is an injective morphism from the $K$--line $\Gamma(A,L)$ into the $K$--vector space
$\Gamma(A,L^{\otimes r^2})$ of dimension $r^{2g}$. We have the following lemma~:
\begin{lem}
\label{racineun}
For any $x, x'\in A_{r^2}(K)$, such that $y=x-x'\in A_r(K)$ there exists a $r^4$--th root of unity
$\mu=\mu_{y}$ in $K$ such that $\varphi_{x'}=\mu\varphi_x$.
\end{lem}
\begin{proof} 
One just need to notice that $(i_{x}'):=(i_{x'})$ is also a good choice, then apply the remarks of \S \ref{isom}.
\end{proof}

Let now $\Gamma$ be any set of representatives in $A_{r^2}(K)$ of $A_{r^2}(K)/A_{r}(K)$. We have
the following~:
\begin{prop}
\label{base}
The map~:
$$\begin{array}{llcl}
\displaystyle\varphi \colon & \displaystyle\Gamma(A,L)^{\Gamma} &\displaystyle
\longrightarrow &\displaystyle \Gamma(A,L^{\otimes r^2})\\
& \displaystyle(s_x)_{x\in\Gamma} &\displaystyle\longmapsto&\displaystyle\sum_{x\in\Gamma}
\varphi_x(s_x)\;,
\end{array}
$$
is an isomorphism of $K$--vector spaces.
\end{prop}
\proof This follows for instance from {\sc Mumford}'s theory of theta structures (see \cite{Mum1} or \cite{DavPhi}). Indeed, the image of $\varphi$ is non zero and invariant under the
irreducible projective representation of the group $K(L^{\otimes r^2})=A_{r^2}(K)$. 

Finally, any bijection $\Gamma\simeq\left\{1,\ldots,r^{2g}\right\}$ provides an isomorphism~:
$$\Gamma(A,L)^{\oplus r^{2g}}\stackrel{\sim}{\longrightarrow} \Gamma(A,L^{\otimes r^2})\;,$$
and therefore, an isomorphism~:
\begin{equation}
\label{proj}
\mathbb{P}(\Gamma(A,L^{\otimes r^2}))\simeq \mathbb{P}_{K}^{r^{2g}-1}.
\end{equation}

The various isomorphisms~(\ref{proj}) obtained by this construction for various choices of the
rigidification of $L$ at the origin, of the coordinate system $(i_x)_{x\in A_{r^2}(K)}$, of
$\Gamma$ and of the bijection $\Gamma\simeq\{1,\ldots,r^{2g}\}$ coincide up to the action of an
element of the finite group~:
$$G={\mathfrak S}_{r^{2g}} \vartriangleleft\mu_{r^4}(K)\subset {\rm GL}_{r^{2g}}(K)$$
(note that above the ${\mathfrak S}$ denotes the group of permutations and not {\sc Siegel}'s upper
half space). See \cite{BiLa} page 168 and \cite{DavPhi} page 654 for more details.

\subsubsection{Theta embeddings}
\label{plong}
The line bundle $L^{\otimes r^2}$ is very ample on $A$ (since $r^2\geq 4>3$). Therefore, it
defines an embedding~:
$$A\longrightarrow \mathbb{P}(\Gamma(A,L^{\otimes r^2}))\;,$$
hence, by composition, with the map~(\ref{proj}), an embedding of $K$--varieties~:
\begin{equation}\label{thetaplonge}
\Theta\colon A\longrightarrow \mathbb{P}_K^{r^{2g}-1}.
\end{equation}

All these constructions are clearly compatible with extensions of the base field $K$ (and in
particular with automorphisms of $K$).

Also observe that if $L'$ is another symmetric ample line bundle on $A$ which defines the same
(principal) polarisation as $L$, then the $\Theta$--embeddings of $A$ into $\mathbb{P}_K^{r^{2g}-1}$
associated to $L$ and $L'$ respectively coincide up to the projective action of the finite group
$G$ defined above. See for example \cite{DavPhi} page 654.

In fact the point $\Theta(0)\in\mathbb{P}^{r^{2g}-1}(K)$ determines up to some finite
ambiguity the
$K$--isomorphism class of the abelian variety $A$ equipped with the polarisation defined by $L$.
Indeed, $\Theta(A)$ can be defined by quadratic equations whose coefficients are functions of the
projective coordinates of $\Theta(0)$, which gives the $\overline{\mathbb{Q}}$--isomorphism class of $A$ but one could have two abelian varieties isomorphic over $\overline{\mathbb{Q}}$ that are not isomorphic over $K$. See \cite{Mum1} and \cite{Mum2}.
\subsubsection{Theta height}

For a projective point $P\in{\mathbb{P}^N(\overline{\mathbb{Q}})}$, we denote by $h(P)$ the $l^2$-logarithmic {\sc Weil} height defined by means of the usual euclidean (or hermitian $l^2$) norms at the infinite places. It is the height with respect to $\mathcal{O}(1)$ equipped with the Fubini-Study metric, \textit{i.e.} (as in our situation $N=r^{2g}-1$) for $P\in\mathbb{P}^{r^{2g}-1}(\overline{\mathbb{Q}})$ it is given by the {\sc Arakelov} degree $h(P)=\widehat{\mathrm{deg}}P^*\overline{\mathcal{O}(1)}_{F.S.}$ of the projective point $P$.

\begin{defin}\label{ThetaHeight} Let $(A,L)$ be a principally polarized abelian variety defined over $\overline{\mathbb{Q}}$, of dimension $g$, with $L$ ample and symmetric. Let $\Theta$ be the projective embedding described in (\ref{thetaplonge}). The Theta height of $A$ with respect to $L$ is then defined as~:
$$h_{\Theta}(A,L):=h(\Theta(0)).$$
\end{defin}

By the preceeding discussion, $h_{\Theta}(A,L)$ depends only on the $\overline{\mathbb{Q}}$--isomorphism
class of $A$ polarized by $L$, and defines a height on the set of such isomorphism classes;
namely, it is bounded below (by $0$) and there is only a finite set, up to
$\overline{\mathbb{Q}}$--isomorphism of pairs $(A,L)$, with bounded height which may be defined over
a number field of bounded degree.

Consider then $K=\mathbb{C}$. Let us fix an homology basis. Let us also fix any embedding of
$\overline{\mathbb{Q}}$ in $\mathbb{C}$. Then (see for example \cite{BiLa} page 213), there exists an element $\tau$ in ${\mathfrak S}_g$ and a point
$z_0\in\mathbb{C}^g$ such that $A(\mathbb{C})$ and $L_{\mathbb{C}}$ may be identified with the complex torus
$\mathbb{C}^g/(\mathbb{Z}^g+\tau\mathbb{Z}^g)$ and ${\mathcal O}(\Theta_{\tau}+[z_0])$ respectively, where~:
$$\Theta_{\tau}=\left\{[z]\in\mathbb{C}^g/(\mathbb{Z}^g+\tau\mathbb{Z}^g),\theta(\tau,z)=0\right\}$$
is the divisor defined by the {\sc Riemann} theta function.

The line bundle $L^{\otimes r^2}$ is very ample and the classical Theta Nullwerte of
the associated embedding can be chosen as the $r^{2g}$ complex numbers~:
$$\theta_{(m_1,m_2)}(\tau,0),\hspace{.5cm}
m_1,m_2\in\left\{0,\frac{1}{r},\ldots,\frac{r-1}{r}\right\}^g\;.$$
Note that certain authors (for instance {\sc Mumford}) rather consider the $\theta_{(m,0)}(r^2\tau,0)$ where
$m\in\frac{1}{r^2}\mathbb{Z}^g/\mathbb{Z}^g$. We shall also make use of this latter coordinate system, but
generally select the former.
They do not all vanish (see {\it e. g.} \cite{Igu}, page~168), and their quotients all
belong to
$\overline{\mathbb{Q}}$ since $A$ is defined over $\overline{\mathbb{Q}}$ (see {\it e. g.}
\cite{Igu}, page~170). Hence, they define some point in
$\mathbb{P}^{r^{2g}-1}(\overline{\mathbb{Q}})$.  The normalized logarithmic {\sc Weil}
height of this point is by definition the theta height
$h_{\Theta}(A,L)$ of the pair
$(A,L)$. Indeed, this height does not depend neither on the choice of the embedding
$\overline{\mathbb{Q}}\hookrightarrow\mathbb{C}$, nor on the choices of $\tau$, $z_0$. It only depends on the
$\overline{\mathbb{Q}}$--isomorphism class of the abelian variety (principally) 
polarized by $L$ (see {\it infra} \S~\ref{thetastruc}).
    The point $\Theta(0)$ is easily seen to have as projective
coordinates the family of Theta Nullwerte $\left(\theta_{(m_1,m_2)}(\tau,0)\right)_{(m_1,m_2)
\in\{0,\frac{1}{r},\ldots,\frac{r-1}{r}\}^{2g}}$ (use the description of the sections of
${\mathcal O}(\Theta_{\tau})$ and of ${\mathcal O}(r^2\Theta_{\tau})$ in terms of theta functions).
This implies that when $A$ is defined over $\overline{\mathbb{Q}}$,  the theta height is also
given by the height of these Nullwerte.

\section{Arakelov geometry of abelian varieties}
\label{baseara}
In this section, we recall a few basic facts concerning the {\sc Arakelov} geometry of
abelian varieties, due essentially to {\sc Moret--Bailly} (see \cite{Moret} and
\cite{Moret2}). They already appear in the present form in \cite{Bost2}, \S~4.~2, to
which we shall refer for proofs and references.
\subsection{Definitions}
\label{defini}
Let $K$ be a number field, ${\mathcal O}_K$ its ring of integers, and
$\pi\colon{\mathcal A}\longrightarrow \Spec({\mathcal O}_K)$ a semi--stable group scheme,
{\it i. e.} a smooth commutative group scheme of finite type and separated over $\Spec({\mathcal
O}_K)$,  such that the components of its fibers are extensions of abelian varieties by tori
(observe that these fibers are not necessarily connected). We shall say that an hermitian line
bundle
${\overline{\mathcal L}}$ on $\mathcal A$ is {\em cubist} if there exists a cubist structure, in the
sense of \cite{Moret}, I.~2.~4.~5, on the $\mathbb{G}_m$--torsor over $\mathcal A$ defined by $\mathcal L$
which, with the notations of {\it loc. cit.}, is defined by a section $\tau$ of ${\mathcal
D}_3({\mathcal L})$ of norm $1$, when ${\mathcal
D}_3({\mathcal L})$ is equipped with the hermitian structure deduced from the one on
$\overline{\mathcal L}$. In other words, if we denote by~:
$$p_i\colon {\mathcal L}^3:={\mathcal L}\times_{{\mathcal O}_K}{\mathcal L}\times_{{\mathcal O}_K}{\mathcal L}\longrightarrow {\mathcal L}\, , \hspace{.5cm} i=1,2,3 $$
the projections on the three factors, by~:
$$ p_I\colon {\mathcal A}^3\longrightarrow{\mathcal A}$$
the morphism which sends a geometric point $(x_1,x_2,x_3)$ to $\sum_{i\in I}x_i$, for any non
empty subset $I$ of $\{1,2,3\}$, and by ${\overline{{\mathcal O}_{{\mathcal A}^3}}}$ the trivial
hermitian line bundle $({\mathcal O}_{{\mathcal A}^3}, \Vert.\Vert)$ defined by $\Vert 1\Vert=1$, then,
an hermitian line bundle $\overline{\mathcal L}$ over $\mathcal A$ is cubist if and only if there exists
an isometric isomorphism~:
\begin{equation}
\label{cubist}
{\mathcal D}_3(\overline{\mathcal L}):=\bigotimes_{I\subset\{1,2,3\},I\neq\emptyset}
\left(p_{I}^{\star}{\overline{\mathcal
L}}\right)^{\otimes(-1)^{\#I}}\stackrel{\sim}{\longrightarrow}\overline{{\mathcal O}_{{\mathcal A}^3}}
\end{equation}
of hermitian line bundles over ${\mathcal A}^3$ which satisfies suitable symmetry and cocycle conditions
({\it confer} \cite{Moret}, I.~2.~4.~5., (i) and (iii)). The relation~(\ref{cubist}) implies that, if 
${\varepsilon}\colon \Spec({\mathcal O}_K)\longrightarrow{\mathcal A}$ denotes the zero section,
$$\varepsilon^{\star}\overline{\mathcal L}\simeq \overline{{\mathcal O}}_{\Spec({\mathcal O}_K)}\;,$$
and also that if ${\mathcal A}_K$ is an abelian variety, the $(1,1)$ form $c_1(\overline{\mathcal L})$ is
translation invariant on each of the complex tori ${\mathcal A}_{\sigma}(\mathbb{C})$, for
$\sigma\colon K\hookrightarrow \mathbb{C}$. Conversely, when $\mathcal A$ is an abelian scheme over
${\mathcal O}_K$, one easily checks that these last two properties characterize cubist hermitian line bundles
over $\mathcal A$.

Let $\pi\colon{\mathcal A}\longrightarrow \Spec({\mathcal O}_K)$ be a semi--stable group scheme whose
generic fiber ${\mathcal A}_K$ is an abelian variety. For any line bundle $\mathcal M$ on $\mathcal A$, the direct
image $\pi_{\star}{\mathcal M}$ is coherent (see \cite{Moret}, Lemma~VI, I.~4.~2) and torsion free,
hence locally free. If $\overline{{\mathcal L}_K}$ is a cubist hermitian line bundle on $\mathcal A$ and if
${\mathcal L}_K$ is ample on ${\mathcal A}_K$, then $\mathcal L$ is ample on $\mathcal A$ (see \cite{Ray},
Theorem VIII.~2, and \cite{Moret}, Proposition~VI.~2.~1) and $c_1({\mathcal L})$ is strictly positive on
${\mathcal A}(\mathbb{C})$ (indeed, it is tranlation invariant on each component of ${\mathcal A}(\mathbb{C})$  and
cohomologous to a strictly positive $(1,1)$ form. Therefore, we may define $\pi_{\star}(\overline{\mathcal
L})$ as the hermitian vector bundle whose rank is~: 
$$\rho({\mathcal L}_{\overline{\mathbb{Q}}}):=\frac{1}{g!}c_1({\mathcal L}_{\overline{\mathbb{Q}}})^g$$
on $\Spec({\mathcal O}_K)$ consisting of $\pi_{\star}({\mathcal L})$ endowed with the hermitian structure
defined by the $L^2$--metric $\Vert.\Vert$ associated to the metric on $\overline{\mathcal L}$ and the
normalized {\sc Haar} measures on the complex tori ${\mathcal A}_{\sigma}(\mathbb{C})$.  In other words, for any
section $s\in\pi_{\star}{\mathcal L}\otimes_{\sigma}\mathbb{C}\simeq {\rm H}^2({\mathcal A}_{\sigma},{\mathcal
L}_{\sigma})$, we let~:
$$\Vert s\Vert_{\sigma}^2=\int_{{\mathcal A}_{\sigma}(\mathbb{C})}\Vert s(x)\Vert^2_{\overline{\mathcal L}}d\mu(x)\;,$$
where $d\mu$ denotes the normalized {\sc Haar} measure on ${\mathcal A}_{\sigma}(\mathbb{C})$. It corresponds to the norm given in \ref{defhautfal}, because the measure is normalized.

\begin{defin}\label{$MB$--model}
Let $A$ be an abelian variety over $\overline{\mathbb{Q}}$, $L$ an ample symmetric line bundle over $A$ and
$F$ a finite subset of $A(\overline{\mathbb{Q}})$. We define a {\em $MB$--model} of $(A,L,F)$ over a number
field $K$ in $\overline{\mathbb{Q}}$ as the data consisting of~:
\begin{itemize}
\item a semi--stable group scheme $\pi\colon {\mathcal A}\longrightarrow \Spec({\mathcal O}_K)$,
\item an isomorphism $i\colon { A}\rightarrow^{\!\!\!\!\!\!\sim} {\mathcal A}_{\overline{\mathbb{Q}}}$ of abelian varieties
over $\overline{\mathbb{Q}}$,
\item a cubist hermitian line bundle ${\overline{\mathcal L}}$ on $\mathcal A$,
\item an isomorphism $\varphi$ as in \ref{base}.
\item for any $P\in F$, a section $\varepsilon_P\colon \Spec({\mathcal O}_K)\longrightarrow
{\mathcal A}$ of the map $\pi$, such that the attached geometric point
namely $\varepsilon_{P,\overline{\mathbb{Q}}}\in {\mathcal A}(\overline{\mathbb{Q}})$ coincides with the point $i(P)$,
\end{itemize}
which satisfy the following condition~: there exists a subscheme $\mathcal K$ of $\mathcal A$, flat and {\em
finite} over $\Spec({\mathcal O}_K)$, such that $i^{-1}({\mathcal K}_{\overline{\mathbb{Q}}})$ coincides with the {\sc
Mumford} group $K(L^{\otimes 2})$, namely the finite algebraic subgroup of $A$ whose rational points
$x$ over $\overline{\mathbb{Q}}$ are characterized by the existence of an isomorphism of line bundles on $A$~:
$$t_x^{\star}L^{\otimes2}\simeq L^{\otimes2}\;.$$
\end{defin}

\begin{defin}\label{bon}
Given a triple $(A,L,r)$ with $A$ an abelian variety over $\overline{\mathbb{Q}}$ and $L$ a symmetric ample line bundle, $r>0$ an even integer, we say that a number field $K$ is \emph{MB} if there exists a $MB$--model of the type $(\pi\colon{\mathcal A}\longrightarrow \Spec({\mathcal O}_K), i, {\overline{\mathcal L}},
\varphi,(\varepsilon_P)_{P\in A_{r^2}})$ rational over $K$.
\end{defin}

\begin{rem}
One can find $MB$ number fields using for example the semi--stable reduction theorem (\textit{confer} \textsc{Moret--Bailly} in \cite{Moret} Theorem 3.5 page 58).
\end{rem}

\subsection{Properties of $MB$--models}
\label{propmb}
The main properties of $MB$--models we shall use in the proof of Theorem~\ref{hautfalhautthet} are essentially due
to {\sc Moret--Bailly} \cite{Moret} and \cite{Moret2}. See also {\sc Breen} \cite{Breen} and {\sc Mumford} \cite{Mum1}. They may be summarized as follows~:

\begin{thm}
\label{prmbmod} Let $A$ be an abelian variety of dimension $g$ over $\overline{\mathbb{Q}}$, $L$ a symmetric ample line
bundle on $A$, and $F$ a finite subset of $A({\overline{\mathbb{Q}}})$. We have the following properties~:
\begin{itemize}
\item[{\rm (i)}] {\em Existence}. For any number field $K_0$, there exist a number field $K$ containing $K_{0}$ and a $MB$--model
$(\pi\colon{\mathcal A}\longrightarrow \Spec({\mathcal O}_K), i, {\overline{\mathcal L}},
\varphi,(\varepsilon_P)_{P\in F})$ for the data $(A,L,F)$.
\item[{\rm (ii)}] {\em {\sc N\'eron--Tate} heights}. For any $MB$--model as in {\rm (i)} and for any $P\in F$, the
normalized height $[K:\mathbb{Q}]^{-1}\widehat{\degr}(\varepsilon^{\star}_P\overline{\mathcal L})$ coincides with the value at
$P$ of the normalized logarithmic N\'eron--Tate height attached to the line bundle $L$ and denoted $\widehat{h}_{L}(P)$.
\item[{\rm (iii)}] {\em Independence of $MB$--models}. For any two $MB$--models
$$(\pi\colon{\mathcal A}\longrightarrow \Spec({\mathcal O}_K), i, {\overline{\mathcal L}},
\varphi,(\varepsilon_P)_{P\in F})$$ and $$(\pi'\colon{\mathcal A}\longrightarrow
\Spec({\mathcal O}_K), i', {\overline{{\mathcal L}'}},
\varphi',(\varepsilon'_P)_{P\in F})$$ of $(A,L,F)$ over a number field $K$, the canonical
isomorphisms defined by
$i$, $\varphi$, $i'$ and $\varphi'$~:
$$(\pi_{\star}{\mathcal L})_{\overline{\mathbb{Q}}}\simeq {\rm H}^0(A,L)\simeq (\pi'_{\star}{{\mathcal L})'}_{\overline{\mathbb{Q}}}$$
and
$$(\varepsilon_P^{\star}{\mathcal L})_{\overline{\mathbb{Q}}}\simeq  L_{\vert P}\simeq ({\varepsilon'_P}^{\star}{{\mathcal
L}'})_{\overline{\mathbb{Q}}}\hspace{1cm} (\forall P\in F)$$ 
extend to isometric isomorphisms of hermitian line bundles over the base $\Spec({\mathcal O}_K)$~:
$$\pi_{\star}({\overline{\mathcal L}})\simeq \pi'_{\star}({\overline{{\mathcal L}'}})$$
and
$${\varepsilon_P}^{\star}({\overline{\mathcal L}})\simeq {\varepsilon'_P}^{\star}({\overline{{\mathcal L}'}})\;.$$
\item[{\rm (iv)}] {\em Compatibility with extensions of scalars}. Let
 $$(\pi\colon{\mathcal A}\longrightarrow \Spec({\mathcal O}_K), i, {\overline{\mathcal L}},
\varphi, (\varepsilon_P)_{P\in F})$$
 be a $MB$--model over some number field $K$, and let $K'$ be
some other number field such that $K\subset K' \subset\overline{\mathbb{Q}}$. From this model, through extension of scalars from
${\mathcal O}_K$ to ${\mathcal O}_{K'}$, we get a semi--stable group scheme~:
$$\tilde{\pi}\colon\tilde{{\mathcal A}}:={\mathcal A}\times_{{\mathcal O}_K}{\mathcal
O}_{K'}\longrightarrow \Spec({\mathcal O}_{K'})\;,$$
an hermitian line bundle $\overline{\tilde{\mathcal L}}$ on $\tilde{\mathcal A}$ (take the pull--back of $\overline{\mathcal L}$
by the first projection ${\mathcal A}\times_{{\mathcal O}_K}{\mathcal O}_{K'}\longrightarrow {\mathcal A}$), and sections~:
$$\tilde{\varepsilon}_P:=\varepsilon_P\otimes_{{\mathcal O}_K}{\mathcal
O}_{K'}\colon\Spec({\mathcal O}_{K'})\longrightarrow\tilde{\mathcal A}\;,$$
and the isomorphisms $i$ and $\varphi$ determine isomorphisms~:
$$\tilde{i}\colon A\rightarrow^{\!\!\!\!\!\!\sim}{\tilde{\mathcal A}}_{\overline{\mathbb{Q}}} \hspace{.5cm} \mbox{\it and}\hspace{.5cm}
\tilde{\varphi}\colon L\rightarrow^{\!\!\!\!\!\!\sim}\tilde{i^{\star}}{\mathcal L}_{\overline{\mathbb{Q}}}\;.$$

The 5--tuple $(\tilde{\pi}\colon \tilde{\mathcal
A}\longrightarrow \Spec({\mathcal O}_{K'}), \tilde{i}, {\overline{\tilde{\mathcal L}}},
\tilde{\varphi}, (\tilde{\varepsilon}_P)_{P\in F})$ is a $MB$--model of $(A,L,F)$ over $K'$.
Moreover, if $j\colon \Spec({\mathcal O}_{K'})\longrightarrow\Spec({\mathcal O}_K)$ denotes the map defined by
the inclusion ${\mathcal O}_K\hookrightarrow{\mathcal O}_{K'}$, then the canonical isomorphism~:
$$j^{\star}\pi_{\star}{\mathcal L}\longrightarrow\tilde{\pi}_{\star}\tilde{\mathcal L}$$
defines an isometric isomorphism of hermitian vector bundles over $\Spec({\mathcal O}_{K'})$~:
$$j^{\star}\pi_{\star}\overline{\mathcal L}\longrightarrow\tilde{\pi}_{\star}\overline{\tilde{\mathcal L}}\;.$$
\item[{\rm (v)}]{\em {\sc Arakelov} slope of $\pi_{\star}\overline{\mathcal L}$}. For any $MB$--model as in {\rm (i)}
one has $\pi_{\star}\overline{\mathcal L}$ semi--stable and~:
$$\frac{\widehat{\degr}\pi_{\star}\overline{\mathcal
L}}{[K:\mathbb{Q}]\rho(L)}=-\frac{1}{2} h_{F}(A)+\frac{1}{4}\log\left(\frac{\rho(L)}{(2\pi)^g}\right).$$
\item[{\rm (vi)}] {\em Base points}. For any $MB$--model as in {\rm (i)}, and any $n\in\mathbb{N}^{\star}$, let ${\mathcal
A}^{[n]}$ be the smallest open subgroup scheme of $\mathcal A$ containing $K({\mathcal L}_{\overline{\mathbb{Q}}}^{\otimes n})$. 
If $n$ is even  and if the closure of $K({\mathcal L}_{\overline{\mathbb{Q}}}^{\otimes n})$ in $\mathcal A$ is finite over
$\Spec({\mathcal O}_K)$, then the global sections ${\rm H}^0({\mathcal A},{\mathcal L}^{\otimes n})$ generate ${\mathcal
L}^{\otimes n}$ over ${\mathcal A}^{[n]}$.
\end{itemize}
\end{thm}
\begin{proof} For details or references concerning the proof of {\rm (i)}--{\rm (v)} see
\cite{Bost2}, \S~4.~3.~2. Assertion {\rm (vi)} follows from \cite{Moret}, VI.~3.4 and
VI.~2.~2.
\end{proof}
\begin{rem}
One can observe that if $(\pi\colon{\mathcal A}\longrightarrow \Spec({\mathcal O}_K), i, {\overline{\mathcal L}},
\varphi,(\varepsilon_P)_{P\in F})$ is a $MB$--model, then $(\pi\colon{\mathcal A}\longrightarrow \Spec({\mathcal O}_K), i, {\overline{\mathcal L}^{\otimes r^2}},
\varphi,(\varepsilon_P)_{P\in F})$ is also a $MB$--model.
\end{rem}

\section{Comparisons of heights}
\label{preuvethun}
\subsection{\kern-2.5ptIntrinsic\kern-1pt{} heights\kern-1pt{} and\kern-1pt{}
projective\kern-1pt{} heights\kern-1pt{} of\kern-1pt{} integral\kern-1pt{} points}
\label{leshautintr}
Let $K$ be a number field, $\pi\colon {\mathcal X}\longrightarrow S:=\Spec({\mathcal O}_K)$
a flat quasi--projective integral scheme such that ${\mathcal X}_K$ is smooth, and $\overline{\mathcal L}$ an hermitian
line bundle on $\mathcal X$.

For any section $P$ of $\pi$, we let as usual~:
$$h_{\overline{\mathcal L}}(P):=\frac{1}{[K:\mathbb{Q}]}\widehat{\degr}P^{\star}\overline{\mathcal L}\;.$$

Let $\overline{\mathcal F}$ be some hermitian vector bundle on $S$ such that ${\mathcal F}\subset\pi_{\star}{\mathcal
L}$ and such that ${\mathcal L}_K$ is generated over ${\mathcal X}_K$ by its global sections in ${\mathcal F}_K\subset
{\rm H}^0({\mathcal X}_K,{\mathcal L}_K)$. The subscheme ${\mathcal B}_{\mathcal F}$ of base points of the linear system
$\mathcal F$ of sections of $\mathcal L$ is defined as the closed subscheme of $\mathcal X$ whose ideal sheaf ${I}_{{\mathcal B}_{\mathcal F}}$ is such that the image of the canonical map~:
$$\pi^{\star}{\mathcal F}\longrightarrow {\mathcal L}$$
is ${I}_{{\mathcal B}_{\mathcal F}}.{\mathcal L}$.

As ${\mathcal B}_{\mathcal F}$ does not meet the generic fiber ${\mathcal X}_K$, for any section $P$ of $\pi$, the
subscheme $P^{\star}{\mathcal B}_{\mathcal F}$ of $\Spec({\mathcal O}_K)$ is a divisor. We shall denote it~:
\begin{equation}\label{betafini}
P^{\star}{\mathcal B}_{\mathcal F}=\sum_{\begin{array}{c}{\mathfrak p} {\rm prime}\\ of {\mathcal
O}_K,\\ {\mathfrak p}\nmid\infty\end{array}}\beta_{\mathfrak p}({\mathcal L},{\mathcal F},P){\mathfrak p}\;.
\end{equation}

The $\beta_{\mathfrak p}({\mathcal L},{\mathcal F},P)$ are non negative integers; almost all of them vanish. They have
archimedian counterparts, defined as follows; for any embedding $\sigma \colon K\hookrightarrow
\mathbb{C}$, we let~:
\begin{equation}\label{betainfini}
\beta_{\sigma}(\overline{\mathcal L}, \overline{\mathcal
F},P):=-\frac{1}{2}\log\left(\sum_{i=1}^{n}\Vert u_i\Vert_{\overline{{\mathcal
L}_\sigma}}^2(P_{\sigma})\right),
\end{equation}
where $(u_i)_{1\leq i\leq n}$ is any orthonormal basis of $\overline{{\mathcal F}_{\sigma}}$, a subspace of
${\rm H}^0({\mathcal X}_{\sigma},{\mathcal L}_{\sigma})$.

We shall denote by $h_{\check{\overline{\mathcal F}}}$ the height on $\mathbb{P}(\check{\mathcal F}_K)$ attached to the
hermitian line bundle ${\mathcal O}_{\check{\overline{\mathcal F}}}(1)$ on the integral model $\mathbb{P}(\check{\mathcal
F})$ of   $\mathbb{P}(\check{\mathcal
F}_K)$. More precisely, ${\mathcal O}_{\check{\overline{\mathcal F}}}(1)$ is ${\mathcal O}_{\check{{\mathcal
F}}}(1)$ equipped with the metric defined by the metric on $\overline{\mathcal F}$ and the canonical
epimorphism $\pi_{\mathbb{P}({\check{\mathcal F}})}^{\star}{\mathcal F}\longrightarrow {\mathcal O}_{\check{{\mathcal
F}}}(1)$.
Let $\nu\colon\widetilde{\mathcal X}\longrightarrow {\mathcal X}$ be the blowing up of ${\mathcal
B}_{\mathcal F}$, and let $E:=\nu^{\star}({\mathcal B}_{\mathcal F})$. It is an effective
vertical {\sc Cartier} divisor on the integral scheme $\widetilde{\mathcal X}$. Let us consider the map $i_K\colon\widetilde{\mathcal X}_K\longrightarrow \mathbb{P}({\mathcal F})_K$.
\begin{prop}
\label{hp}
For any section $P$ of $\pi\colon {\mathcal X}\longrightarrow S$, the following equality holds~:
$$h_{\overline{\mathcal L}}(P)=h_{\overline{\mathcal
F}}(i_K(P_K))+\frac{1}{[K:\mathbb{Q}]}\left(\sum_{{\mathfrak p}\nmid\infty}\beta_{\mathfrak  p}({\mathcal L},{\mathcal F},P)\log(N{
\mathfrak p})+\sum_{\sigma\colon
K\hookrightarrow\mathbb{C}}\beta_{\sigma}(\overline{\mathcal L},\overline{\mathcal
F},P)\right).$$
\end{prop}
\begin{proof} The map
$i_K\colon\widetilde{\mathcal X}_K\longrightarrow \mathbb{P}({\mathcal F})_K$ extends
uniquely to a morphism $i\colon\widetilde{\mathcal X}\longrightarrow \mathbb{P}({\mathcal F})$, by
the very definition of a blowing up. Moreover, the canonical isomorphism of line bundles over
${\mathcal X}_K$~:
$${\mathcal L}_K\simeq i_K^{\star}{\mathcal O}_{\check{\mathcal F}}(1)$$
extends to an isometric isomorphism of hermitian line bundles over $\mathcal X$~:
\begin{equation}
\label{isome}
\nu^{\star}\overline{\mathcal L}\simeq i^{\star}\overline{{\mathcal O}_{\check{\mathcal F}}(1)}\otimes\left({\mathcal
O}(E),\Vert.\Vert\right)
\end{equation}
if $\Vert.\Vert$ denotes the hermitian metric on ${\mathcal O}(E)$ defined by~:
$$\Vert u\Vert_{\sigma}^2=\sum_{i=1}^{n}\Vert u_i\Vert_{\overline{\mathcal L},\sigma}^2$$
on ${\mathcal X}_{\sigma}$, for any orthonormal basis  $(u_i)_{1\leq i\leq n}$ of ${\mathcal F}_{\sigma}$
({\it confer} \cite{Bost2}, 2.~4 and 2.~5). Let $\widetilde{P}\colon S\longrightarrow
\widetilde{\mathcal X}$ the section of
$\widetilde{\pi}:=\pi\circ\nu\colon\widetilde{\mathcal X}\longrightarrow S$ which
lifts $P\colon S \longrightarrow {\mathcal X}$ (it exists by the properness of the map
$\nu\colon \widetilde{\mathcal X}\longrightarrow{\mathcal X}$). We have~:
\begin{equation}
\label{hl} 
[K:\mathbb{Q}]h_{\overline{\mathcal L}}(P)=\widehat{\degr}P^{\star}\overline{\mathcal
L}=\widehat{\degr}\widetilde{P}^{\star}\nu^{\star}\overline{\mathcal L},
\end{equation}
and
\begin{equation}
\label{hlbis}
[K:\mathbb{Q}]h_{\check{\overline{\mathcal F}}}(P_K)=\widehat{\degr}(i\circ\widetilde{P})^{\star}
\overline{{\mathcal O}_{\check{{\mathcal
F}}_K}(1)}=\widehat{\degr}\widetilde{P}^{\star}i^{\star}\overline{{\mathcal O}_{\check{{\mathcal F}}_K}(1)}
\end{equation}
(this follows from the definitions of $h_{\overline{\mathcal L}}$ and $h_{\check{\overline{\mathcal F}}}$).

On the other hand, the arithmetic line bundle $\widetilde{P}^{\star}({\mathcal O}(E),\Vert.\Vert)$ on $S$ is
defined by the arithmetic cycle $(\sum_{{\mathfrak p}\nmid\infty}\beta_{\mathfrak p}({\mathcal L},{\mathcal F}, P){
\mathfrak p},\sum_{\sigma\mid\infty}\beta_{\sigma}({\mathcal L},{\mathcal F},P)\Vert.\Vert_{\sigma})$. Indeed, the multiplicity $\beta_{\mathfrak p}({\mathcal L},\,{\mathcal
F},P)$ defined as the length at $\mathfrak p$ of $P^{\star}{\mathcal B}_K$, coincides with the length at $\mathfrak p$ of $\widetilde{P}^{\star}E$. This implies~:
\begin{equation}
\label{degcha}
\widehat{\degr}\widetilde{P}^{\star}({\mathcal O}(E),\Vert.\Vert)=\sum_{{\mathfrak p}\nmid\infty}\beta_{\mathfrak p}({\mathcal L},{\mathcal F},P)\log(N{\mathfrak p})+\sum_{\sigma\mid\infty}\beta_{\sigma}({\mathcal L},{\mathcal F},P)
\end{equation}

Together with relation~(\ref{degcha}), the relations~(\ref{isome}), (\ref{hl}) and~(\ref{hlbis})
imply the assertion of Proposition~\ref{hp}, which is thus proved.
\end{proof}

We shall also need bounds on the numbers $\beta_{\mathfrak p}({\mathcal L},{\mathcal F},P)$; the following proposition
is useful do derive them~:
\begin{prop}
\label{bornesbeta}
 Let ${\mathcal F}'$ be another vector bundle over $S$ such that~:
$${\mathcal F}\subset {\mathcal F}'\subset\pi_{\star}{\mathcal L}\;,$$
and let $P$ be any section of $\pi$. Then, for any finite prime $\mathfrak p$, we have~:
\begin{equation}
\label{majobeta}
0\leq \beta_{\mathfrak p}({\mathcal L},{\mathcal F}',P)\leq  \beta_{\mathfrak p}({\mathcal L},{\mathcal F},P).
\end{equation}
Moreover, if we further assume that ${\mathcal F}'_K={\mathcal F}_K$, then we have~:
\begin{equation}
\label{majobisbeta}
\sum_{{\mathfrak p}\nmid\infty}m_{\mathfrak p}({\mathcal F},{\mathcal L},P)\log(N{\mathfrak p})-\sum_{{\mathfrak p}\nmid\infty}m_{\mathfrak p}({\mathcal F}',{\mathcal L},P)\log(N{\mathfrak p}) \leq \degr({\mathcal F}'/{\mathcal F}),
\end{equation}
where $\degr({\mathcal F}'/{\mathcal F}):=\log(\#{\mathcal F}'/{\mathcal F})$.
\end{prop}
\begin{proof} If ${\mathcal F}\subset{\mathcal F}'$, then, obviously, ${\mathcal B}_{{\mathcal F}'}\subset{\mathcal B}_{{\mathcal
F}}$. This implies the second inequality of~(\ref{majobeta}). The first comes from the fact that the
$\beta_{\mathfrak p}$'s are lengths.

Let us now prove~(\ref{majobisbeta}) when there exists a finite ideal ${\mathfrak p}_0$ of ${\mathcal
O}_K$ such that ${\mathcal F}'/{\mathcal F}\simeq\mathbb{F}_{{\mathfrak p}_0}$. Then, we have~:
\begin{equation}
\label{unpremier}
{\mathcal B}_{{\mathcal F}'}\cap\pi^{-1}(S-\vert {\mathfrak p}_0\vert)={\mathcal B}_{\mathcal F}\cap\pi^{-1}(S-\vert
{\mathfrak p}_0\vert),
\end{equation}
hence, for every prime ${\mathfrak p}\neq {\mathfrak p}_0$,
$$m_{\mathfrak p}({\mathcal L},{\mathcal F},P)=m_{\mathfrak p}({\mathcal L},{\mathcal F}',P)\;,$$
and the relation~(\ref{majobisbeta}) amounts to the bound~:
$$\beta_{{\mathfrak p}_0}({\mathcal L},{\mathcal F},P)\leq  \beta_{{\mathfrak p}_0}({\mathcal L},{\mathcal F}',P)+1\;.$$
Indeed, we are going to show the following inclusion of ideal sheaves~:
\begin{equation}
\label{incluideaux}
{\mathfrak I}_{{\mathcal B}_{{\mathcal F}'}}.{\mathfrak I}_{{\mathcal X}_{{\mathfrak p}_0}}\subset {\mathfrak I}_{{\mathcal B}_{\mathcal
F}},
\end{equation}
where ${\mathcal X}_{{\mathfrak p}_0}$ denotes the scheme theoretic fiber of ${\mathfrak p}_0$ in $\mathcal X$
({\it i. e.} a vertical {\sc Cartier} divisor on $\mathcal X$).
Let $s\in{\mathcal F}'$ whose class in ${\mathcal F}'/{\mathcal F}\simeq\mathbb{P}_{{\mathfrak p}_0}$ does not vanish, and
let $\sigma\in{\rm H}^0({\mathcal X},{\mathcal L})$ the corresponding global section of $\mathcal L$ on $\mathcal X$.
Choose $\alpha\in{\mathfrak p}_0\backslash {\mathfrak p}_0^2$. According to the definition of both ${\mathcal
B}_{\mathcal F}$ and ${\mathcal B}_{{\mathcal F}'}$, we have the following equality of subsheaves of $\mathcal L$~:
$${\mathfrak I}_{{\mathcal B}_{\mathcal F}}{\mathcal L}+{\mathcal O}_{\mathcal X}\sigma={\mathfrak I}_{{\mathcal B}_{{\mathcal F}'}}{\mathcal L}\;.$$
Moreover, $\alpha s\in {\mathcal F}$, therefore, $\alpha \sigma$ is a section of ${\mathfrak I}_{{\mathcal
B}_{\mathcal F}}.{\mathcal L}$ and~:
\begin{equation}
\label{inclu}
\alpha{\mathfrak I}_{{\mathcal B}_{{\mathcal F}'}}{\mathcal L}\subset {\mathfrak I}_{{\mathcal B}_{\mathcal F}}{\mathcal L}.
\end{equation}
This proves that the inclusion~(\ref{incluideaux}) holds in a neighbourhood of ${\mathcal X}_{{\mathfrak p}_0}$, hence on ${\mathcal X}$ itself by the relation~(\ref{unpremier}).

The general case of the inequality~(\ref{majobisbeta}) follows from the special case we have just
proven, by considering a maximal strictly increasing chain $({\mathcal F}_i)_{0\leq i\leq n}$ of
submodules~:
$${\mathcal F}={\mathcal F}_0\subsetneq{\mathcal F}_1\subsetneq\cdots\subsetneq{\mathcal F}_n={\mathcal F}'\;,$$
and applying the inequality~(\ref{majobisbeta}) to ${\mathcal F}={\mathcal F}_{i-1}$, ${\mathcal F}'={\mathcal F}_i$,
 for $i$ varying between $1$ and $n$, and adding the inequalities thus obtained.
Such a chain exists by the {\sc Jordan}--{\sc Hölder} theorem applied to ${\mathcal F}'/{\mathcal F}$, each
quotient ${\mathcal F}_i/{\mathcal F}_{i-1}$ is isomorphic to $\mathbb{F}_{{\mathfrak p}_i}$ for some prime ${\mathfrak p}_i$
of ${\mathcal O}_K$ by maximality of the chain, and finally, one has just to remark that~:
$$\sum_{i=1}^{n}\log(N{\mathfrak p}_i)=\log(\#{\mathcal F}'/{\mathcal F})\;.$$
Proposition~\ref{bornesbeta} is thus completely established.
\end{proof}

\section{Theta embeddings and Arakelov geometry}
\label{thetaara}

\subsection{Height of points}
As in \S~\ref{propmb}, consider an abelian variety $A$ of dimension $g$ defined over $\overline{\mathbb{Q}}$ and
$L$ a symmetric ample line bundle over $A$ defining a principal polarisation of $A$, together with a strictly
positive even integer $r$.

According to Theorem~\ref{prmbmod}, (i), there exists some $MB$ number field $K$ and some $MB$--model
$(\pi\colon{\mathcal A}\longrightarrow\Spec({\mathcal O}_K),i,\overline{\mathcal
L},\varphi,(\varepsilon_x)_{x\in A_{r^2}(\overline{\mathbb{Q}})})$ of
$(A,L,A_{r^2}(\overline{\mathbb{Q}}))$ such that any point $x\in A_{r^2}(\overline{\mathbb{Q}})$ is rational over
$K$, and extends to some section of $\pi$, which we will still denote by the same letter $x$.

Let~:
$$j\colon [r]^{\star}{\mathcal L}_K\rightarrow^{\!\!\!\!\!\!\sim} {\mathcal L}_K^{\otimes r^2}$$ 
be the isomorphism of line bundle over ${\mathcal A}_K$ defined in subsection~\ref{isom} (since ${\mathcal L}_K$
is cubist, it is automatically rigidified), and let~:
$$ i_x\colon t_x^{\star}{\mathcal L}_K^{\otimes r^2}\rightarrow^{\!\!\!\!\!\!\sim}{\mathcal L}_K^{\otimes
r^2},\hspace{.5cm} x\in {\mathcal A}_{K, r^2}(K)$$
be a good choice of isomorphisms, as in \S~\ref{isom}. Then, we have~:
\begin{prop}
\label{iso}
We have the following properties~:
\begin{itemize}
\item[{\rm (i)}] The isomorphisms $j$ and $i_x$ extend to isometric isomorphisms of hermitian line
bundles, which we will still denote by the same letters~:
$$j\colon [r]^{\star}\overline{\mathcal L}\rightarrow^{\!\!\!\!\!\!\sim}\overline{\mathcal L}^{\otimes r^2}\;,$$
and~:
$$i_x\colon t_x^{\star}\overline{\mathcal L}^{\otimes r^2}\rightarrow^{\!\!\!\!\!\!\sim}\overline{\mathcal L}^{\otimes
r^2}\;.$$
\item[{\rm (ii)}] The maps $\varphi_x$ and $\varphi$ (see Proposition \ref{base}) extend to isometric maps of
hermitian line bundles~:
$$\varphi_x\colon\pi_{\star}\overline{\mathcal L}\longrightarrow \pi_{\star}\overline{{\mathcal
L}}^{\otimes r^2}\;,$$
and~:
$$ \varphi=(\varphi_x)_{x\in\Gamma}\colon (\pi_{\star}{\overline{\mathcal L}})^{\Gamma}
\longrightarrow \pi_{\star}\overline{{\mathcal L}}^{\otimes r^2}$$
(where $(\pi_{\star}\overline{\mathcal L})^{\Gamma}$ is the direct sum of $r^{2g}$--copies
of the hermitian line bundle $\pi_{\star}\overline{\mathcal L}$).
\end{itemize}
\end{prop}
\begin{proof} The existence of a cubist structure on $\overline{\mathcal L}$ implies (copy the usual
arguments) that $j$ extends as an hermitian isometric isomorphism between $[r]^{\star}\overline{\mathcal
L}\simeq\overline{\mathcal L}^{\otimes r^2}$ and that there exist isometric hermitian isomorphisms~:
$$\widetilde{i_x}\colon t_{x}^{\star}\overline{\mathcal L}^{\otimes r^2}\rightarrow^{\!\!\!\!\!\!\sim}
\overline{\mathcal L}^{\otimes r^2}, \hspace{.5cm} x\in{\mathcal A}_{r^2}({\mathcal O}_K)\;.$$

For any two points $x$ and $y$ in ${\mathcal A}_{r^2}({\mathcal O}_K)$,
$\widetilde{i_{x+y}}\circ(\widetilde{i_x}\circ t_x^{\star}\circ\widetilde{i_y})^{-1}$ is an isometric
hermitian automorphism of $\overline{\mathcal L}^{\otimes r^2}$, therefore, the multiplication by some
root of unity. This implies that each $\widetilde{i_x}$ coincides with $i_x$ up to some root of unity
(simply define $\lambda_x\in K^{\star}$ by $i_x=\lambda_x\widetilde{i_x}$ and observe that  the map
$x\longmapsto [\lambda_x]$ defines a morphism from the finite abelian group $A_{r^2}(K)$ to the
torsion free group $K^{\star}/\mu_{\infty}(K)$). Therefore, $i_x$ extends to an isometric hermitian
isomorphism as claimed. This proves part {\rm (i)}.

Let us now prove part (ii). The fact that the $\varphi_x$'s and $\varphi$ extend to maps of ${\mathcal
O}_K$--modules follows from the fact that $j$ and $i_x$ extend to morphisms of schemes over $\mathcal A$.
The fact that these extensions are isometric in turn implies that $\varphi_x$ is isometric. Finally,
$\varphi$ is isometric since, for any pair $(x,x')\in\Gamma^2$, $x\neq x'$, and any embedding
$\sigma\colon K\hookrightarrow \mathbb{C}$, the images of $\varphi_{x,\sigma}$ and
$\varphi_{x',\sigma}$ are orthogonal in $\pi_{\star}\overline{\mathcal L}^{\otimes r^2}_{\sigma}$; this
follows for instance (see \cite{DavPhi} page 656) from the orthogonality of the classical theta functions with characteristics
$\theta_{(m_1,m_2)}(\tau,rz)$, for $(m_1,m_2)$ varying in $\{0,\frac{1}{r},\ldots,\frac{r-1}{r}\}^2$.
This completes the proof of part (ii) and thus of the Proposition~\ref{iso}.
\end{proof}

Let now $\overline{\mathcal F}:=(\pi_{\star}\overline{\mathcal L})^{\Gamma}$. By means of the
map $\varphi$, $\overline{\mathcal F}$ may be identified with a submodule of $\pi_{\star}{\mathcal
L}^{\otimes r^2}$. Any bijection~:
$$\Gamma\simeq \{1,\ldots, r^{2g}\}\;,$$
determines an isomorphism~:
$$\mathbb{P}(\check{\mathcal F})\simeq\mathbb{P}_{{\mathcal O}_K}^{r^{2g}-1}\;.$$

Moreover, as $\overline{{\mathcal O}_{\check{\mathcal F}}(1)}\simeq \pi^{\star}\pi_{\star}\overline{\mathcal
L}\otimes\overline{{\mathcal O}(1)}$, the usual {\sc Weil} (logarithmic and absolute) height $h$ and $h_{\overline{\mathcal F}}$
verify, for any point  $P\in \mathbb{P}(\check{\mathcal
F})(\overline{\mathbb{Q}})\simeq\mathbb{P}^{r^{2g}-1}(\overline{\mathbb{Q}})$~:
$$h(P)=h_{\overline{\mathcal F}}(P)-\frac{1}{[K:\mathbb{Q}]}\widehat{\degr}\pi_{\star}\overline{\mathcal L}\;.$$
Therefore, by Theorem~\ref{prmbmod}, (v), we get~:
\begin{equation}
\label{egahaut}
h(P)=h_{\overline{\mathcal F}}(P)+\frac{1}{2} h_F(A)+\frac{g}{4}\log(2\pi).
\end{equation}

Finally, if we apply Proposition~\ref{hp} to the data $({\mathcal A},\overline{\mathcal L}^{\otimes r^2},
\overline{\mathcal F})$, in place of $({\mathcal X},\overline{\mathcal L},\overline{\mathcal F})$, and if we
use~(\ref{egahaut}) and Theorem~\ref{prmbmod}, part (ii), and if we observe that the morphism
$i_K\colon{\mathcal A}_K\longrightarrow\mathbb{P}({\mathcal F}_K)$ coincide with the theta embedding
$\Theta\colon {\mathcal A}_K\longrightarrow\mathbb{P}_K^{r^{2g}-1}$ of subsection~\ref{plong}, we get~:

\begin{lem}
\label{hnt}
Let $(A,L)$ be a principally polarized abelian variety over a number field $K$, of dimension $g$ and level $r$, with $L$ symmetric and ample. For any section $P$ of $\pi\colon {\mathcal A}\longrightarrow\Spec({\mathcal O}_K)$, we have~:
$$\begin{array}{lcl}\displaystyle
\widehat{h}_{L}(P) & = & \displaystyle h(\Theta(P_K))-\frac{1}{2} h_F(A)-\frac{g}{4}\log(2\pi)\\ & + & 
\displaystyle\rule{0mm}{8mm}
\frac{1}{[K:\mathbb{Q}]}\kern-1pt\left(\kern-1.2pt\sum_{{\mathfrak p}\nmid\infty}\beta_{\mathfrak p}({\mathcal
L}^{\otimes r^2},{\mathcal F}, P)\log(N{\mathfrak p})
+\kern-1.2pt\sum_{\sigma\colon k\hookrightarrow
\mathbb{C}}\kern-1pt\beta_{\sigma}(\overline{\mathcal L}^{\otimes r^2},\overline{\mathcal F},P)\kern-1.2pt\right).
\end{array},$$
where $\beta_{\mathfrak{p}}$ and $\beta_{\sigma}$ are defined by the equations (\ref{betafini}) and (\ref{betainfini}).
\end{lem}

\subsection{Bounds for the contribution of the base points}
\label{borptbase}
Thanks to Theorem~\ref{prmbmod} and to Proposition~\ref{bornesbeta}, it is easy to bound the contribution of the base
points over finite places. More precisely, we get the~:
\begin{prop}
\label{borptbaseprop}
For any $P\in{\mathcal A}({\mathcal O}_K)$, the following inequalities hold~:
\begin{itemize}
\item[{\rm (i)}] for any prime ideal $\mathfrak p$ ($\neq 0$) of ${\mathcal O}_K$ one has~:
$$ 0\leq \beta_{\mathfrak p}({\mathcal L}^{\otimes r^2},\pi_{\star}{\mathcal L}^{\otimes r^2},P)\leq\beta_{\mathfrak
p}({\mathcal L}^{\otimes r^2},{\mathcal F},P)\;;$$
\item[{\rm (ii)}] the difference of multiplicities is also bounded as follows~:
\[
\frac{1}{[K:\mathbb{Q}]}\sum_{{\mathfrak p}\nmid\infty}\left(\beta_{\mathfrak p}({\mathcal L}^{\otimes r^2},{\mathcal F},P)-\beta_{\mathfrak p}({\mathcal
L}^{\otimes r^2},\pi_{\star}{\mathcal L}^{\otimes r^2},P)\right)\log(N{\mathfrak p}) \leq \frac{g}{2}r^{2g}\log(r)\;;
\]

\item[{\rm (iii)}] Moreover, for any $\mathfrak p$, if the component (over $\mathbb{F}_{\mathfrak p}$) of ${\mathcal
A}_{\mathbb{F}_{\mathfrak p}}$ containing $P_{\mathbb{F}_{\mathfrak p}}$ meets\footnote{{\it i. e.} if some element $x$ of ${\mathcal A}({\mathcal O}_K)_{r^2}$ is such that
$x_{\mathbb{F}_{\mathfrak p}}$ and $P_{\mathbb{F}_{\mathfrak p}}$ lie in the same component of ${\mathcal A}_{\mathbb{F}_{\mathfrak p}}$} the closure in $\mathcal A$ of ${\mathcal
A}_{K,r^2}(K)$,
then~:
\begin{equation}
\label{multinulle}
\beta_{\mathfrak p}({\mathcal L}^{\otimes r^2},\pi_{\star}{\mathcal L}^{\otimes r^2},P)=0.
\end{equation}
In particular, the relation~(\ref{multinulle}) holds if $P$ is the zero section $\varepsilon$ of
$\mathcal A$.
\end{itemize}
\end{prop}
\begin{proof} the point (i) follows from relation~(\ref{majobeta}), and the point (ii) from the relation~(\ref{majobisbeta})
and from Theorem~\ref{prmbmod} part~(v), which shows that~:
$$\begin{array}{lcl}\displaystyle
\widehat{\degr}\left(\pi_{\star}{\mathcal L}^{\otimes r^2}/{\mathcal
F}\right) & = & \displaystyle\widehat{\degr}\pi_{\star}\overline{\mathcal L}^{\otimes r^2}
-\widehat{\degr}\overline{\mathcal F} 
=  \widehat{\degr}\pi_{\star}{\mathcal L}^{\otimes
r^2}-r^{2g}\widehat{\degr}\pi_{\star}\overline{\mathcal L} \\ & = &\kern-1pt
\rule{0mm}{7mm}\displaystyle r^{2g}\left(\frac{1}{4}
\log(\rho({\mathcal L}^{\otimes r^2}))-\frac{1}{4}\log(\rho({\mathcal L}))\right)= r^{2g}\frac{1}{4}\log(r^{2g})
\end{array}$$
Finally, equation~(\ref{multinulle}) follows from Theorem~\ref{prmbmod}, part~(vi).
Proposition~\ref{borptbaseprop} is thus proved.
\end{proof}

We now turn to the archimedean counterparts of the $\beta_{\mathfrak p}$. They are easily expressed in
terms of the classical theta functions (also compare with \cite{Bost1}, Appendix C). We summarize
the estimates we need in~:
\begin{lem}
\label{borptbaselem}
Let $(A,L)$ be a principally polarized abelian variety over a number field $K$, of dimension $g$ and level $r$, with $L$ symmetric and ample. Let $P$ be a point of ${\mathcal A}({\mathcal O}_K)$ and $\sigma\colon K\hookrightarrow\mathbb{C}$ a
complex embedding. Let $\tau_{\sigma}$ be a point in ${\mathfrak S}_g$ (the {\sc Siegel} space) such that~:
\begin{equation}
\label{choixtau}
{\mathcal A}_{\sigma}(\mathbb{C})\simeq\mathbb{C}^g/(\mathbb{Z}^g+\tau_{\sigma}\mathbb{Z}^g)
\end{equation}
as principally polarized abelian varieties, and let $z\in\mathbb{C}^g$ be such that
$[z]\in\mathbb{C}^g/(\mathbb{Z}^g+\tau_{\sigma}\mathbb{Z}^g)$ is the image of $P_{\sigma}$ by the
map~(\ref{choixtau}). Then, we have~:
\begin{equation}
\label{borbetasig}
\beta_{\sigma}(\overline{\mathcal L}^{\otimes r^2},\overline{\mathcal F},P)
=-\frac{1}{2}\log\left(2^{\frac{g}{2}}\sum_{e\in{\mathcal
Z}_r(\tau_{\sigma})}\Vert\theta\Vert^2(\tau_{\sigma},rz+e)\right),
\end{equation}
where we denote by ${\mathcal Z}_r(\tau_{\sigma})$ the set
$\frac{1}{r}(\mathbb{Z}^g+\tau_{\sigma}\mathbb{Z}^g)/(\mathbb{Z}^g+\tau_{\sigma}\mathbb{Z}^g)$.
\end{lem}
The right hand side of the equation~(\ref{borbetasig}) may be bounded by using the following estimates
which are also of independent interest~:
\begin{prop}
\label{bornormthet}
\label{estthetaldeux}
We use the notations ${\mathfrak S}_{g}$ for the {\sc Siegel} space of principally polarized abelian varieties and ${\mathfrak F}_{g}$ for the fondamental domain. We have the following inequalities~:
\begin{itemize}
\item[{\rm (i)}] For any $\tau\in{\mathfrak S}_g$, 
$$\max_{e\in{\mathcal Z}_r(\tau)}\{\Vert
\theta\Vert^2(\tau,e)\}\geq\left(\det(\Ima\tau)\right)^{\frac{1}{2}}\;.$$
\item[{\rm (ii)}] For any $\tau\in {\mathfrak F}_{g}$, and any $z\in\mathbb{C}^g$, we have~:
$$\Vert\theta\Vert^2(\tau,z)\leq c(g)\left(\det(\Ima\tau)\right)^{\frac{1}{2}}\;,$$
where $c(g)$ denotes a constant which depends only on $g$. We can take for instance~:
$$c(g)=\left(2+\frac{2}{3^{\frac{1}{4}}}2^{\frac{g^3}{4}}\right)^g\;.$$
\end{itemize}

In particular, for any $\tau\in {\mathfrak F}_{g}$,
one has~:
\[
\frac{g}{4}\log(2) \leq\frac{1}{2}\log\Big(2^{\frac{g}{2}}\sum_{e\in{\mathcal
Z}_r(\tau)}\Vert\theta\Vert^2(\tau,e)\Big)-\frac{1}{4}\log(\det(\Ima\tau))\leq\frac{1}{2}\log c(g)+\frac{g}{4}\log(2)+g\log(r)\;.
\]
\end{prop}
\begin{proof} The point (i) is equivalent to the assertion~:
$$\forall \tau\in{\mathfrak S}_g,\kern14pt F(\tau)=\max_{(m_1,m_1)\in\{0,\frac{1}{2}\}^{2g}}\left\{\left\vert
\theta_{(m_1,m_2)}(\tau,0)\right\vert\right\}\geq 1\;.$$
This follows from the duplication formula which shows that $F(\tau)\geq F(2\tau)$, and from the observation~:
$$\lim_{n\rightarrow\infty}\theta(2^n\tau,0)=1\;,$$
compare with \cite{Dav1}, \S~3.~3 to \S~3.~5.

The point (ii) is Lemma~3.~4 of \cite{Dav1}, with the explicit constant found in the work of {\sc Graftieaux} by combining equation (14) page 101 and equation (17) page 103 of \cite{Graf}.
\end{proof}

\subsubsection{Proof of Theorem~\ref{hautfalhautthet}}\label{ProofMainThm}

To complete the proof of Theorem~\ref{hautfalhautthet}, it is now enough to apply Lemma~\ref{hnt} and the results of the former section to the point $P=\varepsilon$, the zero section of $\pi\colon{\mathcal
A}\longrightarrow\Spec({\mathcal O}_K)$. In that way we get the inequalities with 
$$M(r,g)=\frac{g}{4}\log(4\pi)+g\log(r)+\frac{1}{2}\log c(g)\;,$$
and~:
$$m(r,g)=\frac{g}{4}\log(4\pi)-\frac{g}{2}r^{2g}\log(r)\;.$$

Applied to an arbitrary section $P\in{\mathcal A}({\mathcal O}_K)$, the above estimates also give the
following comparison between the {\sc Weil} height $h(P)$ and the {\sc N\'eron}--{\sc Tate} height $\widehat{h}_{L}(P)$ of the point $P$~:

\begin{thm}
\label{mazaint}
Let $(A,L)$ be a principally polarized abelian variety over a number field $K$, of dimension $g$ and level $r$, with $L$ symmetric and ample. We denote by $\tau_\sigma$ the period matrix in the fundamental domain $\frak{F}_g$ for the archimedian place $\sigma$. For any point $P\in A(\overline{\mathbb{Q}})$, we have~:
$$\widehat{h}_{L}(P)\geq h(\Theta(P))-\frac{1}{2} h_F(A)-\frac{1}{4[K:\mathbb{Q}]}\sum_{\sigma\colon
K\hookrightarrow\mathbb{C}}\log(\det(\Ima\tau_{\sigma}))-C(r,g)\;,$$
where one can take $C(r,g)=\frac{g}{4}\log(4\pi)+g\log(r)+\frac{g}{2}\log\left(2+\frac{2}{3^{1/4}}2^{g^3/4}\right)$.
\end{thm}
It should be observed that any point $P\in A(\overline{\mathbb{Q}})$ is integral and extends to a suitable
$MB$--model; the machinery can then be used, since the degree of the number field on which the $MB$--model
is defined does not interfere in the estimates.

To obtain Corollary \ref{varcomphaut} it then suffices to apply Theorem \ref{hautfalhautthet} for part $(1)$ and $(3)$ (using Definition \ref{modifiedFaltHeight}). For part $(2)$, it suffices to use part $(1)$ with $\max\{h_{\Theta}(A,L),1\}$ and $\max\{h_{F}(A),1\}$ in the left hand side of the inequality, plus the following easy lemma~:

\begin{lem}
Let $a\geq 1$ and $b\geq 1$ be real numbers. Suppose that there exists a number $c\geq 2$ such that $|a-b|\leq c\log(2+a)$ (we will refer to this inequality by $(*)$ along the proof). Then we have $|a-b|\leq \tilde{c} \log(2+\min\{a,b\})$, where one can choose $\tilde{c}=c\log(6+2c\log(2c)-2c)/\log(3)$.
\end{lem}
\begin{proof}
Let $g(x)=c\log(2+x)-x/2$. Then for all $x\geq 1$ one has $g(x)\leq g(2c-2)$. Thus~:
\[
c\log(2+a)\leq \frac{a}{2}+c\log(2c)-c+1\;,
\]
hence using $(*)$~:
\[
a\leq b+c\log(2+a)\leq b+\frac{a}{2}+c\log(2c)-c+1\;,
\]
then~: $a\leq 2b+2c\log(2c)-2c+2.$
We get in $(*)$~:
\[
|a-b|\leq c\log(2+a)\leq c\log\Big(4+2b+2c\log(2c)-2c\Big)\;.
\]
One can show the inequality, valid for all $y\geq 3$ and $d\geq 0$~:
\[
\log(2y+d)\leq \frac{\log(6+d)}{\log(3)}\log(y)\;.
\]
One gets with $y=2+b$ and $d=2c\log(2c)-2c$~:
\[
|a-b|\leq c\frac{\log(6+2c\log(2c)-2c)}{\log(3)}\log(2+b)\;.
\]
As $\log(6+2c\log(2c)-2c)/\log(3)\geq 1$, it gives the lemma.

\end{proof}

Finally, to get a proof of Proposition~\ref{ExplicitBound}, use Proposition 3.7 page 527 of {\sc Rémond} \cite{Rem2} and the explicit bounds of \cite{DavPhi} of pages 662 and 665 to complete the estimate. A similar computation has been done in \cite{Paz} pages 116--117 in the case of jacobians of genus 2 curves.

\section{Comparison of differential lattices}\label{diff lattice}

We will study in the following several differential lattice structures associated to an abelian variety.

\subsection{Integral forms}

We consider $\Lie(A)^{\check{\phantom{a}} }=\Omega_{A,0}^1$. Given a triple $(A,L,r)$ with $A$ an abelian variety, $L$ a symmetric ample line bundle associated to a principal polarisation and $r>0$ an even integer, and given a $MB$ field $K$ for this triple (see definition \ref{bon}), we will study the following $\mathcal{O}_{K}$--integral forms of $\Omega_{A,0}^1$. We denote by ``d'' the differential operator, which we normalize such that for any non--zero sections $s_{1}$ and $s_{2}$ we have $s_{2}^{\otimes 2}d(s_{1}/s_{2})$ integral over $\mathcal{O}_{K}$. See for instance \cite{Graf} page 107.
\begin{enumerate}
\item The \textsc{Néron} lattice $\mathcal{N}=\varepsilon^{*}\Omega_{\mathcal{A}/S}^1$.
\item The big \textsc{Shimura} lattice, defined as follows : let $\theta \in{\Gamma(A,L)\backslash \{0\}}$ and $\Gamma$, $\varphi_{x}$, \textit{etc.} be as in paragraph \ref{thetastruc}. Let $\theta_{x}=\varphi_{x}(0)$. The family $(\theta_{x})_{x\in{\Gamma}}$ is a base over $K$ of $\Gamma(A,L^{\otimes r^2})$. Then the big \textsc{Shimura} lattice is :
\[
\mathcal{S}h=\sum_{\substack{(x,x')\in{\Gamma^2}\\ \theta_{x}(0)\neq 0}}\mathcal{O}_{K} d\Big(\frac{\theta_{x'}}{\theta_{x}}\Big)(0)\;.
\]
\item The small \textsc{Shimura} lattice : let $\underline{x}=(x_{0},...,x_{g})\in{\Gamma^{g+1}}$ such that $\theta_{x_{0}}(0)\neq 0$ (hence $\theta_{x_{0}}$ is even) and such that the differentials $\Big(d(\theta_{x_{i}}/\theta_{x_{0}})(0)\Big)_{1\leq i \leq g}$ is a $K$--base of $\Omega_{A,0}^1$. We let then :
\[
\mathcal{S}h_{\underline{x}}=\sum_{i=1}^{g}\mathcal{O}_{K}d\Big(\frac{\theta_{x_{i}}}{\theta_{x_{0}}}\Big)(0)\;.
\]
\item Let $K$ be a $MB$ field for the triple $(A,L,r)$ and $(\pi\colon{\mathcal A}\rightarrow\Spec({\mathcal O}_K),i,\overline{\mathcal
L},\varphi,(\varepsilon_x)_{x\in A_{r^2}})$ the associated model. We call ``abstract \textsc{Shimura} differential'' (see \cite{Bost1} page 795--28) the morphism of $\mathcal{O}_{K}$--modules :
\[
\Sigma : (\pi_{*}\mathcal{L}^{\otimes r^2})^{\otimes 2}\to \varepsilon^* \Omega_{\mathcal{A}/S}^{1}\;.
\]

\end{enumerate}

\begin{lem}
Let $\mathcal{N}$ and $\mathcal{S}h$ be the lattices defined previously. Then :
\begin{enumerate}
\item These lattices only depend on $(A,L,r)$ and $K$.
\item Let $K$ be a $MB$ field for $(A,L,r)$ and let $K'/K$ be a finite extension. Then $K'$ is also MB for $(A,L,r)$. Moreover, if $\mathcal{N}'$ and $\mathcal{S}h'$ are respectively the $\mathcal{O}_{K'}$--\textsc{Néron} lattice and the $\mathcal{O}_{K'}$--\textsc{Shimura} lattice associated to $(A,L,r)$, we have the canonical isomorphisms $\mathcal{N}'\simeq \mathcal{N}\otimes_{\mathcal{O}_{K}}\mathcal{O}_{K'}$ and $\mathcal{S}h\simeq \mathcal{S}h\otimes_{\mathcal{O}_{K}}\mathcal{O}_{K'}$.
\end{enumerate}
\end{lem}

\begin{proof}
Follows from the definition \ref{bon}. The canonical isomorphism is deduced by the commutativity of the following diagram :

\dgARROWLENGTH=0.5cm %longueur fleche

\[
\begin{diagram}
\node{\mathcal{N}'\otimes_{\mathcal{O}_{K'}}K'}
\arrow[3]{e,t}{\sim}
\arrow[3]{s,r}{\lbag}
\node[3]{\mathcal{N}\otimes_{\mathcal{O}_{K}}K'}
\arrow[3]{s,r}{\lbag}\\\\\\
\node{\Omega_{A/K'}^{1}(0)}
\arrow[3]{e,t}{\sim}
\node[3]{\Omega_{A/K}^{1}(0)\otimes_{K}K'}
\end{diagram}
\]

\end{proof}

\subsection{Comparison of integral forms}

\subsubsection{A distance between the lattices on a $K$--vector space}\label{distance}
Let $K$ be a number field and $V$ a $K$--vector space of dimension $g$. Let us consider the set :
\[
\mathcal{R}(V)=\Big\{\mathcal{V}\subset V \;\Big|\;\mathcal{V}\; \textrm{sub--}\mathcal{O}_{K}\textrm{--module\;free\;of\;finite\;type\;generating\;V\;over\;} K\Big\}\;.
\]

For all $(\mathcal{V}_{1},\mathcal{V}_{2})\in{\mathcal{R}(V)^2}$ we set :
\[
\delta(\mathcal{V}_{1},\mathcal{V}_{2})=\frac{1}{[K:\mathbb{Q}]}\log\Card\Big((\mathcal{V}_{1}+\mathcal{V}_{2})/\mathcal{V}_{1}\cap\mathcal{V}_{2}\Big)\;.
\]

\begin{prop}
The function $\delta$ is a distance on $\mathcal{R}(V)$.
\end{prop}

\begin{proof}
We have easily that for any $\mathcal{V}_{1}$ and $\mathcal{V}_{2}$ in $\mathcal{R}(V)$, $\delta(\mathcal{V}_{1},\mathcal{V}_{2})=\delta(\mathcal{V}_{2},\mathcal{V}_{1})$. Moreover, if $\delta(\mathcal{V}_{1},\mathcal{V}_{2})=0$, then $(\mathcal{V}_{1}+\mathcal{V}_{2})/\mathcal{V}_{1}\cap\mathcal{V}_{2}=\{0\}$, hence any element of $\mathcal{V}_{1}$ is in $\mathcal{V}_{2}$ and vice versa. Let $\mathcal{V}_{1}$, $\mathcal{V}_{2}$ and $\mathcal{V}_{3}$ be in $\mathcal{R}(V)$, and $v_{1}\in{\mathcal{V}_{1}}$, $v_{3}\in{\mathcal{V}_{3}}$. Pick any $v_{2}\in{\mathcal{V}_{2}}$, then the equalities of the type $v_{1}+v_{3}=v_{1}+v_{2}-v_{2}+v_{3}$ give the inclusion :
\[
\Big((\mathcal{V}_{1}+\mathcal{V}_{3})/\mathcal{V}_{1}\cap\mathcal{V}_{3}\Big)\subset \Big((\mathcal{V}_{1}+\mathcal{V}_{2})/\mathcal{V}_{1}\cap\mathcal{V}_{2}\Big)\;+\;\Big((\mathcal{V}_{2}+\mathcal{V}_{3})/\mathcal{V}_{2}\cap\mathcal{V}_{3}\Big)\;.
\]
One just needs to bound from above the cardinality of the right hand side to get the triangular inequality for $\delta$, which is easy.

\end{proof}

\begin{rem}
Suppose $\mathcal{V}_{1}\subset \mathcal{V}_{2}$. Then $[K:\mathbb{Q}]\delta(\mathcal{V}_{1},\mathcal{V}_{2})=\log\Card(\mathcal{V}_{2}/\mathcal{V}_{1})$ is just the index of a sublattice.
\end{rem}

If $K'/K$ is a finite extension, let $V'=V\otimes_{K}K'$ and let $\mathcal{R}(V')$ denote the $\mathcal{O}_{K'}$--lattices. We get an injection :
\begin{align*}
i\colon \mathcal{R}(V) & \to \mathcal{R}(V')\\
\mathcal{V} & \mapsto \mathcal{V}\otimes_{\mathcal{O}_{K}}\mathcal{O}_{K'}.
\end{align*}

\begin{prop}
Let $\delta'$ be the distance on $\mathcal{R}(V')$ defined as above. Then :
\[
\forall (\mathcal{V}_{1},\mathcal{V}_{2})\in{\mathcal{R}(V)^2},\; \; \delta'(i(\mathcal{V}_{1}),i(\mathcal{V}_{2}))=\delta(\mathcal{V}_{1},\mathcal{V}_{2})\;.
\]
\end{prop}

\begin{proof}
One just needs to apply $[K':\mathbb{Q}]=[K':K][K:\mathbb{Q}]$ in the definition of $\delta'$.
\end{proof}

In this setting, we will now show the following statement :

\begin{thm}\label{thm diff}
Let $g\geq 1$ and $r>0$ an even integer. There exists a constant $c(g,r)>0$ such that for any triple $(A,L,r)$ with $A$ of dimension $g$, for any associated $MB$ number field $K$, for any $\underline{x}\in{\Gamma}$ defining a small \textsc{Shimura} lattice, one has :
\[
\max\Big\{\delta(\mathcal{N},\mathcal{S}h),\,\delta(\mathcal{N},\mathcal{S}h_{\underline{x}}),\, \delta(\mathcal{N},\mathrm{im}\Sigma)\Big\}\leq \; c(g,r)\, \max\{1,h_{\Theta}(A)\}\;,
\]
and one can take $c(g,r)=4+8C_{2}+g\log(\pi^{-g}g!e^{\pi r^2}g^4)+4r^{2g}$, where $C_{2}$ is given in Corollary~\ref{varcomphaut}.

\end{thm}

\begin{proof}

As we have $\delta(\mathcal{N},\mathcal{S}h)\leq \delta(\mathcal{N},\mathrm{im}\Sigma)+\delta(\mathrm{im}\Sigma,\mathcal{S}h)$ and $\delta(\mathcal{N},\mathcal{S}h_{\underline{x}})\leq \delta(\mathcal{N},\mathcal{S}h)+\delta(\mathcal{S}h_{\underline{x}},\mathcal{S}h)$, it suffices to upper bound the three quantities $\delta(\mathcal{N},\mathrm{im}\Sigma)$, $\delta(\mathrm{im}\Sigma,\mathcal{S}h)$ and $\delta(\mathcal{S}h_{\underline{x}},\mathcal{S}h)$.

\vspace{0.2cm}

We begin by $\delta(\mathcal{N},\mathrm{im}\Sigma)$. By definition, one has $\mathrm{im}\Sigma\subset \mathcal{N}$, so we are in fact trying to bound the index of a sublattice. We use the notation $\overline{\mathcal{N}}$ and $\overline{\mathrm{im}\Sigma}$ for the lattices considered with the hermitian structure given by the Riemann form associated to $L$. Then we have :
\[
\delta(\mathcal{N},\mathrm{im}\Sigma)=\widehat{\degr}(\overline{\mathcal{N}})-\widehat{\degr}(\overline{\mathrm{im}\Sigma})\;.
\]

We then use the point (v) in Theorem \ref{prmbmod} to estimate the slope of $\pi_{*}\mathcal{L}$ and the slope inequality of \cite{Bost1} Proposition 4.3 page 795--15 to get :
\[
\delta(\mathcal{N},\mathrm{im}\Sigma)\leq 2h_{F}(A)-\frac{1}{2}\log\left(\frac{r^{2g}}{(2\pi)^g}\right)+\frac{1}{[K:\mathbb{Q}]}\sum_{\sigma:K\hookrightarrow \mathbb{C}}\log\Vert\Sigma\Vert_{\sigma}\;.
\]
Use then the inequality of \cite{Bost1} page 795--29. One can precise the constant denoted $C_{27}$ by using Lemma 5.8 page 795--25 combined with the estimate on the ``rayon d'injectivit\'e'' of \cite{DavPhi}, Lemma 6.8 page 698, to get :
\[
\frac{1}{[K:\mathbb{Q}]}\sum_{\sigma:K\hookrightarrow \mathbb{C}}\log\Vert\Sigma\Vert_{\sigma}\leq \left(\frac{g}{2}\log(\pi^{-g}g!e^{\pi r^2}g^4)\right)\log(2+h_{\Theta})\;.
\]

Note that a similar estimate has been obtained in \cite{Graf} equation (24) page 108, with the {\sc Faltings} height instead of the Theta height.
\vspace{0.2cm}

We now estimate $\delta(\mathcal{N},\mathcal{S}h)$. As explained in \cite{Bost1} page 795--28, one has :

\begin{align*}
\Sigma:\pi_{*}\mathcal{L}^{\otimes r^{2}}\times \pi_{*}\mathcal{L}^{\otimes r^{2}}& \to \pi_{*}\mathcal{L}^{\otimes 2r^{2}}\otimes \Omega^{1}_{A/K} \to \Omega^{1}_{A/K,0}\\
s_{1}\otimes s_{2}\quad\quad & \mapsto s_{2}^{\otimes 2}d(s_{1}/s_{2})\quad \quad\mapsto s_{2}^{\otimes 2}d(s_{1}/s_{2})|_{0}
\end{align*}
Thus we need to clear out the denominators of $\mathcal{S}h$ in exactly the same way as done in this definition of $\Sigma$ ; it suffices to multiply by $\prod\theta_{x}(0)^2$, where the product is taken over all $x\in{\Gamma(A,L^{\otimes r^2})}$ such that $\theta_{x}(0)\neq 0$. We then roughly upper bound :
\[
\delta(\mathcal{N},\mathcal{S}h)\leq 2r^{2g}h_{\Theta}(A)\;.
\]

\vspace{0.2cm}

We finally give the estimation of $\delta(\mathcal{S}h_{\underline{x}},\mathcal{S}h)$. We have $\mathcal{S}h_{\underline{x}}\subset \mathcal{S}h$, so, clearing out the denominators as above :
\[
\delta(\mathcal{S}h_{\underline{x}},\mathcal{S}h)=\frac{1}{[K:\mathbb{Q}]}\log\Card\Big(\mathcal{S}h/\mathcal{S}h_{\underline{x}}\Big)\leq 2r^{2g} h_{\Theta}(A)\;.
\]

We can conclude by using Corollary \ref{varcomphaut} to explicitely compare $h_{F}(A)$ and $h_{\Theta}(A)$.

\end{proof}

We give the following easy lemma to get the last corollary of Theorem \ref{thm diff} :

\begin{lem}
Let $a\geq1$, $b\geq1$, $c>0$ and $d\in{\mathbb{R}}$. If $|a-b|\leq c\log(2+\min\{a,b\})$ and $d\leq a$, then :
\[
d\leq (1+2c)\min\{a,b\}\;.
\]
\end{lem}
\begin{proof}
Just write $d\leq a \leq b+c\log(2+\min\{a,b\})\leq b+c\log(2+b)\leq b+2cb.$
\end{proof}

\begin{cor}
Let $g\geq 1$ and $r>0$ an even integer. There exists a constant $c(g,r)>0$ such that for any triple $(A,L,r)$ with $A$ of dimension $g$, for any associated $MB$ number field $K$, for any $\underline{x}\in{\Gamma}$ defining a small \textsc{Shimura} lattice, one has :
\[
\max\Big\{\delta(\mathcal{N},\mathcal{S}h),\,\delta(\mathcal{N},\mathcal{S}h_{\underline{x}}),\, \delta(\mathcal{N},\mathrm{im}\Sigma)\Big\}\leq \; \Big(1+2c(g,r)\Big)\, \min\{h_{\Theta},h_{F}\}\;,
\]
and one can take $c(g,r)=4+8C_{2}+g\log(\pi^{-g}g!e^{\pi r^2}g^4)+4r^{2g}$, where $C_{2}$ is given in Corollary~\ref{varcomphaut}.

\end{cor}

\end{document}